\documentclass{article}

\usepackage{graphicx}
\usepackage{amsmath}
\usepackage{amsthm}
\usepackage{amssymb}
\usepackage{hyperref}
\usepackage{todonotes}
\usepackage{cleveref}

\usepackage[toc,page]{appendix}

\theoremstyle{plain}
\newtheorem{theorem}{Theorem}[section]

\newtheorem{lemma}[theorem]{Lemma}

\newcommand{\Mod}[1]{\ (\mathrm{mod}\ #1)}

\newcommand{\Fam}[1]{\ensuremath{\mathcal{F}_{#1}}}

\def\Z{\mathbb Z}
\def\s{\sigma}
\def\<{\langle}
\def\>{\rangle}

\DeclareMathOperator{\Gal}{Gal}
\newcommand{\absgal}{\Gal\left(\overline{\mathbb{Q}}/\mathbb{Q}\right)}
\newcommand{\algnums}{\overline{\mathbb{Q}}}

\usepackage[utf8]{inputenc}
\usepackage{amsthm}
\newtheorem{definition}{Definition}

\usepackage{authblk}

\title{Pairs of tree dessins, their Shabat polynomials, and monodromy groups}
\author[1]{Benjamin Dupont\thanks{\texttt{bd@up8.edu}}}
\author[1]{Revekka Kyriakoglou\thanks{\texttt{kyriakoglou@up8.edu}}}
\author[2]{Vassilis Metaftsis\thanks{\texttt{vmet@aegean.gr}}}
\author[3]{Efstratios Prassidis\thanks{\texttt{efprass@math.uoa.gr}}}
\author[1]{Alexandros Singh\thanks{\texttt{as@up8.edu}}}

\affil[1]{Université Paris 8, France}
\affil[2]{University of the Aegean, Greece}
\affil[3]{National and Kapodistrian University of Athens, Greece}

\date{}

\begin{document}

\maketitle

\begin{abstract}
Coverings of the Riemann sphere by itself, ramified over two points,
are given by so-called Shabat polynomials. The correspondence
between Grothendieck's dessins d'enfants and Belyi maps then implies
a bijection between Shabat polynomials and tree dessins (bicolored plane trees). 
Dessins can be assigned a combinatorial invariant known as their 
passport, which records the degrees of their vertices.
We consider all possible passports determining a pair of tree dessins,
determining the associated Shabat polynomials and monodromy groups.
\end{abstract}

\section{Introduction}

A theorem of Belyi \cite{belyui1980galois} asserts that a Riemann surface $X$
is an algebraic curve defined over $\algnums$ if and only if there exists a
{\em Belyi map}: a holomorphic map $f: X \rightarrow \mathbb{P}^1$ onto the
Riemann sphere, that is ramified over at most three points. 
Grothendieck, fascinated by this result he describes as ``profond et
deroutant'' \cite{grothendieck1984esquisse}, demonstrated a bijection between
the set of isomorphism classes of Belyi maps and that of {dessins d'enfants}:
connected bicoloured maps (in the sense of graphs drawn on surfaces).

The potential of studying surfaces/curves defined over $\algnums$ and the
action of the absolute Galois group $\absgal$ on them via a combinatorial
object proved very alluring, spawning an active area of research, see for
example
\cite{shabat2007drawing,schneps1994grothendieck,girondo2012introduction,
lando2013graphs,guillot2015elementary,jones2016dessins,neumann2020galois}
for some overviews.

In this work, we focus on the case of covers $f: \mathbb{P}^1 \rightarrow
\mathbb{P}^1$ of the sphere by itself, ramified over at most 2 points. Such
Belyi maps are polynomials with at most two critical values known as {\em
Shabat polynomials} and the corresponding dessins, which we can recover as
$f^{-1}([c_1,c_2])$ where $[c_1,c_2]$ is an interval connecting the two
critical values $c_1, c_2$, are (weighted, bicolored,
plane) {\em trees}. Such polynomials are interesting from various points of view.
They are a particular and relatively simple case of a Belyi map which can often
be calculated via elementary means as we'll soon see. However,
they demonstrate non-trivial analytic, geometric, and arithmetic behaviour. For example, they provide a ``canonical geometric form'' for every plane tree. 
Furthermore these forms are conformally balanced as shown in \cite{bishop2014true} 
and can be approximated using a numerical conformal welding procedure
known as the zipper algorithm (see \cite{marshall2007convergence,barnes2014conformal}).

Perhaps most important is the fact the group $\absgal$ acts faithfully on trees
\cite{schneps1994dessins} and various invariants have been studied that can
help us understand said action, some combinatorial, others of a more arithmetic
nature, see \cite{zapponi2000fleurs,zvonkin2019pell} for examples of the last. 
An important invariant is a tree's {\em passport}: a triple
$[\alpha;\beta;n]$ of integer partitions $\alpha,\beta \vdash n$ in which
$\alpha$ records the degrees of black vertices and $\beta$ those of white ones.
Following the terminology established in \cite{lando2013graphs}
we'll call the set of dessins sharing a given passport $P$ a {\em family}. Note
that such a family doesn't not necessarily correspond to an orbit under the
action of $\absgal$: a family can split into several
$\absgal$-orbits.

Another invariant of interest is the {\em monodromy group} of the covering,
corresponding to the {\em cartographic group} of a dessin. To define it, we
first must label the $n$ edges of a dessin with an integer in $[n]$, in a
unique manner. We then define a tuple of permutations $\sigma_0,\sigma_1$, 
each with as many cycles as there are black/white vertices respectively, which
record the combinatorial data corresponding to the embedding of a graph on
surface:
\begin{itemize}
    \item Each cycle of $\sigma_0$ corresponds to the cyclic ordering of edges
        around its corresponding black vertex.
    \item Each cycle of $\sigma_1$ records the ordering of edges 
        around its corresponding white vertex.
\end{itemize}
Notice how the partitions $\alpha,\beta$ of the passport of a dessin correspond
to the cycle structures of $\sigma_0, \sigma_1$ respectively. Finally,
the cartographic group is $\langle \sigma_0, \sigma_1 \rangle \leq S_n$, 
which is conjugate in $S_n$ to the monodromy group of the covering
generated by the corresponding Belyi map.
We note in passing that the monodromy group of a Belyi map has another
description which is of particular interest in the context of inverse Galois
theory: it is the Galois group of the smallest Galois covering of which its a
quotient.


Building databases of Belyi maps and their associated invariants is an active
area of research. A catalog of tree dessins with at most 8 edges and their
corresponding Shabat polynomials compiled by Bétréma and Zvonkin is available
in \cite{betremaZvonkinList}. Tree dessins with nine and ten edges are listed
in \cite{kochetkov2009plane,kochetkov2014short}. For pairs of dessins, see
\cite{adrianov2022calculating}. More generally, the L-functions and modular
forms database \cite{lmfdb} contains a database of Belyi maps and their
monodromy groups (see, also, \cite{musty2019database}). A classification of
passports yielding families of size $1$ appears in \cite{adrianov2009plane}.
The corresponding Shabat polynomials and monodromy groups where calculated in
\cite{adrianov2020davenport,cameron2019shabat}. The goal of our work is to
extend this catalogue to all families of size 2, which are the smallest cases
in which we can observe the action of $\absgal$ producing something other than
fixed-points. These families are exhaustively listed in
\cite{adrianov2009plane} as well as in \cite{lando2013graphs} where the list is
attributed to D. P\'er\'e. They consist of 6 infinite families (3 of diameter 4
and 3 of diameter 6):

\begin{enumerate}
    \item $[r,s,t;3,1^{n-3};n]$ where $n=r+s+t$, \hfill $\mathbf{(F_1)}$
    \item $[r^2, s^2;4,1^{n-4};n]$ where $n = 2r + 2s$, \hfill $\mathbf{(F_2)}$
    \item $[r^3,s^2;5,1^{n-5};n]$ where $n=3r+2s$, \hfill $\mathbf{(F_3)}$
    \item $[r,s,1^t;3^p;n]$ where $n = 3p = r+s+t$ and $p+t+2 = n+1$, \hfill $\mathbf{(F_4)}$
    \item $[r^2,1^s;4^p; n]$ where $4p=2r+s$ and $p+s+2=n+1$, \hfill $\mathbf{(F_5)}$
    \item $[r^2,1^s;5^p; n]$ where $n=5p=2r+s$ and $p+s+2=n+1$, \hfill $\mathbf{(F_6)}$
\end{enumerate}
as well as six pairs deemed {\em sporadic}
\begin{enumerate}
    \item $[3^2,1;2^2,1^3;7]$, \hfill $\mathbf{(F_7)}$
    \item $[3 2^2;2^2,1^3;7]$, \hfill $\mathbf{(F_8)}$
    \item $[3^2,1^3;2^4,1;9]$, \hfill $\mathbf{(F_9)}$
    \item $[3,2^2,1^3;2^5;10]$, \hfill $\mathbf{(F_{10})}$
    \item $[4^3,1^8;2^{10};20]$, \hfill $\mathbf{(F_{11})}$
    \item $[5^3,1^{11};2^{13};26]$. \hfill $\mathbf{(F_{12})}$
\end{enumerate}
The two trees for each family above are depicted in Figure \ref{fig:alltrees}.
\begin{figure}
    \centering
    \includegraphics[width=1\linewidth]{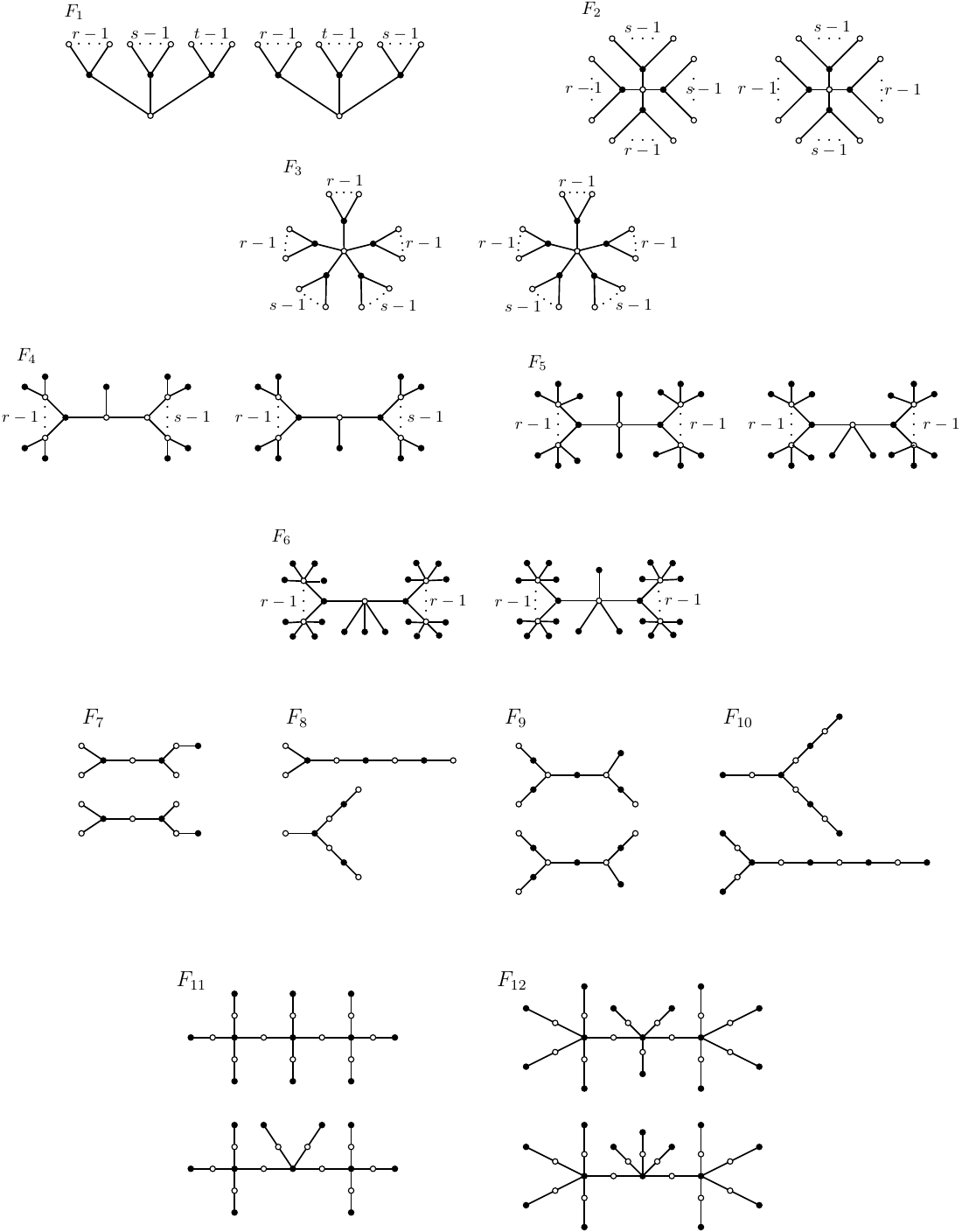}      
    \caption{The families of passports with exactly two trees.}
    \label{fig:alltrees}
\end{figure}

\section{Preliminaries}

As mentioned in the introduction, 
a Shabat polynomial $P(x): \mathbb{P}^1 \rightarrow \mathbb{P}^1$ is
a polynomial with at most two finite critical values. 
We consider such polynomials up to automorphisms of the source and target spheres which fix infinity,
that is to say, up to affine transformations ``inside and outside'',
as follows.
\begin{definition}[Equivalence of Shabat polynomials]
Let $P,Q$ be two Shabat polynomials with critical values 
$c_{1,1},c_{1,2}$ and $c_{2,1},c_{2,2}$ respectively.
We say that $P,Q$ are {\em equivalent} if there exist constants
$A,B,a,b \in \mathbb{C}$ with $a \neq 0$ such that 
\begin{equation*}
\begin{split}
    Q(x) &= AP(ax+b)+B, \\ 
    c_{2,1} &= Ac_{1,1} + B, \\ 
    c_{2,2} &= Ac_{1,2} + B.
\end{split}
\end{equation*}
\end{definition}

There exists, as shown in \cite{lando2013graphs} for example, a bijection between the set of {\em bicoloured, plane trees} and 
Shabat polynomials taken up to equivalence.
The absolute Galois group $\absgal$ acts on Shabat polynomials by conjugation
of their coefficients, inducing an action on their corresponding dessins.
The following lemma guarantees that the coefficients of the Shabat polynomial
associated to a dessin can be chosen so as to belong to its {\em field of moduli}: the extension of $\mathbb{Q}$ corresponding to its stabiliser in $\absgal$.
\begin{lemma}[\cite{couveignes1994calcul}]
The Shabat polynomial of a tree $T$ can be defined over the field of moduli of $T$.
\end{lemma}
Indeed, for a given tree dessin $T$, the intersection of all fields $K \subseteq \algnums$ over which its corresponding polynomial $P$ can be defined is its field of moduli.

Shabat polynomials can be composed provided we ``align'' the critical values of one with the vertex coordinates of the other:
\begin{lemma}
Let $P(x),Q(x)$ be two Shabat polynomials with critical values
$c_{1,1},c_{1,2}$ and $c_{2,1},c_{2,2}$ respectively. Then if $P(c_{1,1}),
P(c_{1,2}) \in \{c_{2,1},c_{2,2}\}$, the composition $P(Q(x))$ is Shabat.
\end{lemma}
The effect on the corresponding dessins is described in \cite{zvonkine2011functional,adrianov1998composition} and intuitively amounts
to marking two vertices of $P$ and replacing paths between black and white vertices in the dessin of $Q$ with copies of $P$ in such a way that
an edge in $Q$ gets ``expanded'' into the unique path between the two marked vertices of $Q$.  The following theorem of Ritt is useful for detecting such compositions. 
\begin{theorem}[\cite{ritt1922prime}]
A polynomial ramified covering is a composition of two or
more coverings of smaller degrees if and only if its monodromy group is imprimitive.  
\end{theorem}

Trees with primitive monodromy groups where classified in \cite{adrianov2013weighted},
see also \cite{adrianov2009plane} where the possible composition factors of monodromy groups
for unicellular dessins are classified.

\section{Family \texorpdfstring{\Fam{1}}{F_1}}

Recall that $P_1$ is the passport $[r,s,t;3,1^{k};n]$ where $k=r+s+t-3$.
The integers $r$, $s$ and $t$ must be pairwise distinct, otherwise the family
would not contain two distinct trees. Figure \ref{F:family1} shows the
``canonical geometric form`` of the two distinct trees in the case $r = 3$,
$s=5$ and $t=6$. Both trees in \Fam{1} are defined over $\mathbb{Q}(\sqrt{-rst(r+s+t)})$
which always is a quadratic extension, given that $-rst(r+s+t)$ is always
a negative integer.

\begin{figure}[ht]
    \centering
    \begin{minipage}[b]{0.48\textwidth}
        \includegraphics[width=\linewidth]{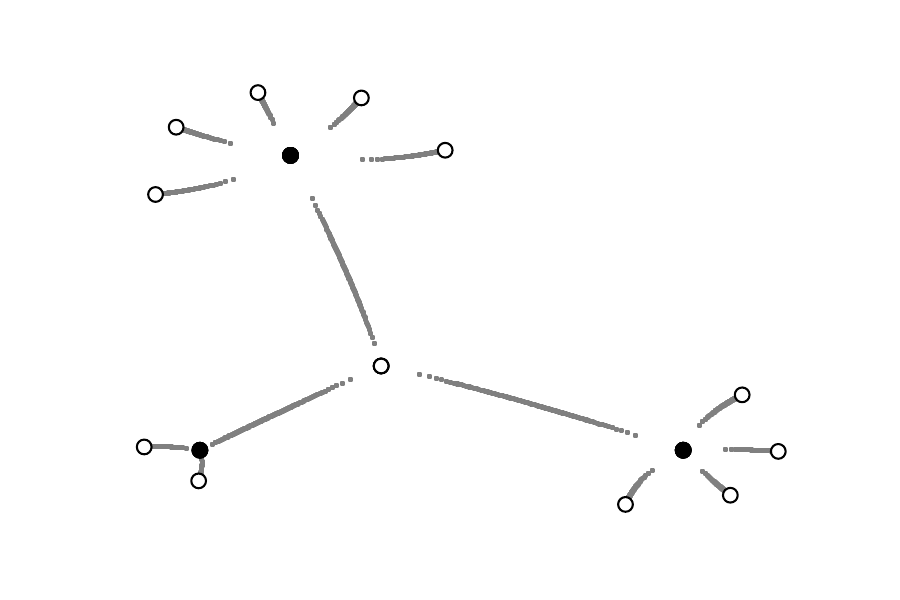}
    \end{minipage}
    \hfill
    \begin{minipage}[b]{0.48\textwidth}
        \includegraphics[width=\linewidth]{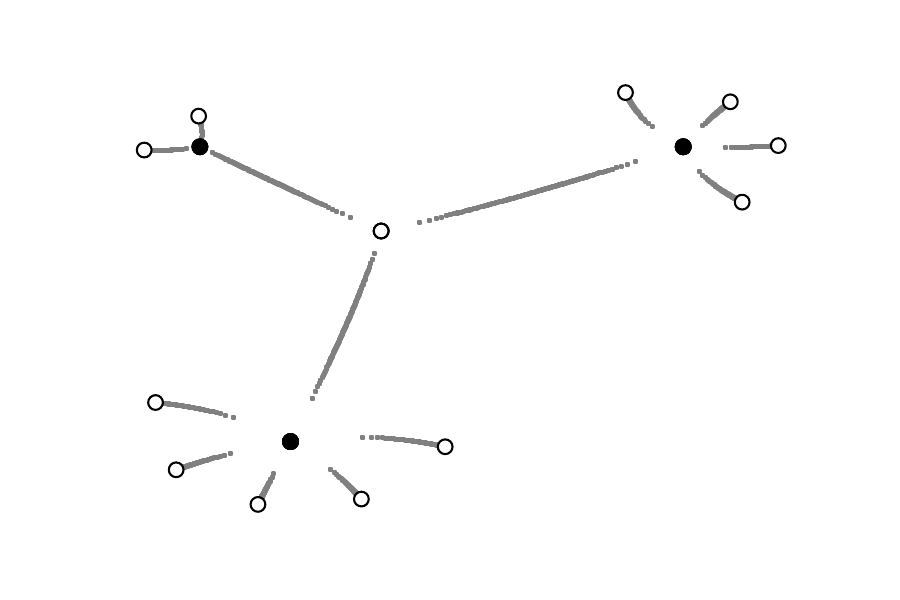}
    \end{minipage}
    \caption{Trees in \Fam{1} with $r=3,s=5$ and $t=6$.}
    \label{F:family1}
\end{figure}

\subsection{Shabat polynomial}

\begin{lemma}\label{lem:p2}
The Shabat polynomials for the combinatorial family \Fam{1} 
are given by
\begin{equation}
    P(x) = x^r (x-1)^s (x-a)^r
\end{equation}
where
\[ a = \frac{r^2+rs+rt+st \mp 2 \sqrt{-rst(r+s+t)}}{(r+s)^2}.
\]
\end{lemma}

\begin{proof}
Let us fix two positions for the first two black vertices (with respective degree $r$ and $s$) at $0$ and $1$. Denote by $a$ the position of the remaining black vertex. Then, the associated Shabat polynomial will have the form
\[ P(x) = x^r (x-1)^s (x- \alpha)^t. \]
The derivative of $P(x)$ is then
\begin{equation}\label{eq:factoredDeriveF1}
P'(x) = x^{r-1} (x-1)^{s-1} (x-\alpha)^{t-1} [r(x-1)(x-\alpha) + sx(x-\alpha) + tx(x-1) ].
\end{equation}

Note that both of the roots $0$ and $1$ of $P'(x)$ map to the same
critical point $P(0)=P(1)=0$.
Since we want $P(x)$ to be Shabat, we need to make sure that there's
at most one other critical value different than $0$. 
To do so, we let $Q(x)$ be the remaining factor of \cref{eq:factoredDeriveF1}
i.e $Q(x) = r(x-1)(x-\alpha) + sx(x-\alpha) + tx(x-1)$. Since $Q(x)$ is
of degree 2 we can arrange for it to have a double root $\rho$ by choosing
$a$ in a way such that the discriminant of $Q$ is $0$. This guarantees that there's at most one more critical value $P(\rho)$ for $P(x)$ and yields the desired formula for the two possible values of $a$. 
\end{proof}

\subsection{Monodromy Groups}
One can can choose the labels on the tree in such a way that the two generators of the monodromy group are
\begin{align*}
\sigma_0 & = (1, \ldots, r)(r+1, \ldots, r+s) (r+s+1, \ldots, r+s+t) \\
\sigma_1 & = (1, r+1, s+r+1).
\end{align*}
Then $\sigma_{\infty}=\sigma_0 \sigma_1 = (1 , \ldots, r+s+t)$.

\begin{lemma}
Let $n=r+s+t$ and $d=\gcd(r,s,t)$. The monodromy group of the trees associated with the passport $P_1$ is 
\begin{equation*}
\begin{cases}
(A_{\frac nd})^d \rtimes \Z_{2d} \text{ if $n/d$ even}\\
(A_{\frac nd})^d\rtimes \Z_d \text{ if $n/d$ odd}
\end{cases}
\end{equation*}
\end{lemma}

\begin{proof}
Let $H$ to be the subgroup of $G$ generated by $(\s_{\infty})^i\s_1(\s_{\infty})^{-i}$, $i=0,\ldots, n$. Since $\s_{\infty}$ and $\s_1$ generate $G$, it is obvious that $H$ is normal in $G$. Moreover, the elements of $H$ can be partitioned into $d$ subsets each of which contains $n/d$ elements. These are $\{(i,r+1+i,s+r+1+i), (i+d,r+1+i+d,s+r+1+i+d), \ldots, (i+(n/d-1)d, r+1+i+(n/d-1)d, r+s+1+(n/d-1)d)\}$ for $i=1,\ldots,d$.

One can easily see that each $D_i$ contains $n/d$ 3-cycles and so $\<D_i\>\cong A_{n/d}$. Hence, $H$ is the direct product of $d$ subgroups, i.e $H=(A_{n/d})^d$. Morever, $G/H$ is cyclic of order $2d$ if $n/d$ is even and of order $d$ if $n/d$ is odd. The result follows.
\end{proof}

\section{Family \texorpdfstring{\Fam{2}}{O_2}}

The family \Fam{2}, for fixed and distinct values of $r$ and $s$, 
consists of two trees which are distinguished by the manner 
in which their four black vertices are cyclically arranged around 
the unique degree-$4$ white vertex. 
We'll refer to these two trees as $T_{2,1}$ and $T_{2,2}$ with the convention
being that:
\begin{itemize}
    \item The tree $T_{2,1}$ is the one in which recording the degrees
    of the black vertices around the white one yields the cycle $(r,s,r,s)$.
    \item The tree $T_{2,2}$ is the one in which recording the degrees
    of the black vertices around the white one yields the cycle $(r,r,s,s)$.
\end{itemize}
The family \Fam{2} splits into two Galois fixed
points for all valid values of $r,s$. Figure \ref{F:family2} shows the two distinct trees in the case $r = 3$ and $s=5$.

\begin{figure}[ht]
    \centering
    \begin{minipage}[b]{0.48\textwidth}
        \includegraphics[width=\linewidth]{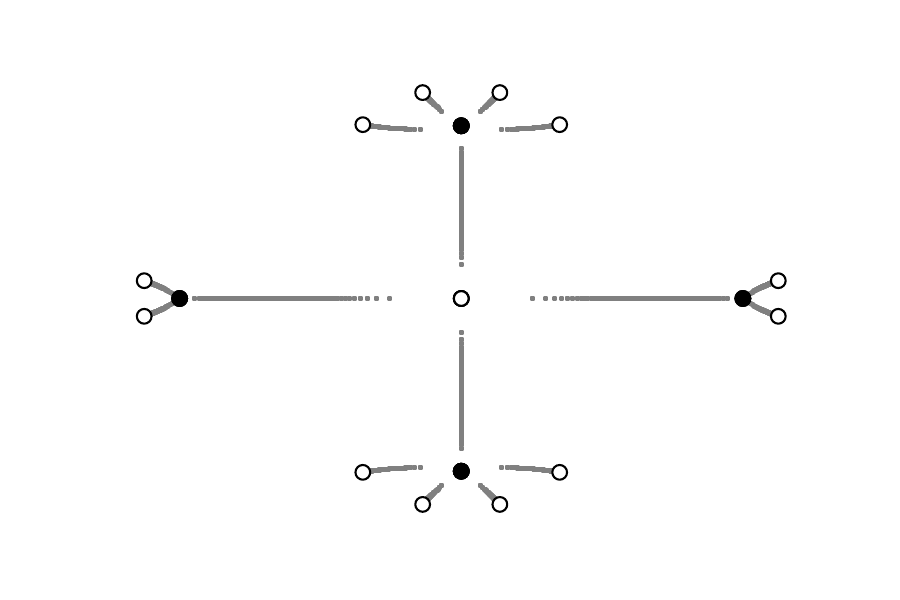}
    \end{minipage}
    \hfill
    \begin{minipage}[b]{0.48\textwidth}
        \includegraphics[width=\linewidth]{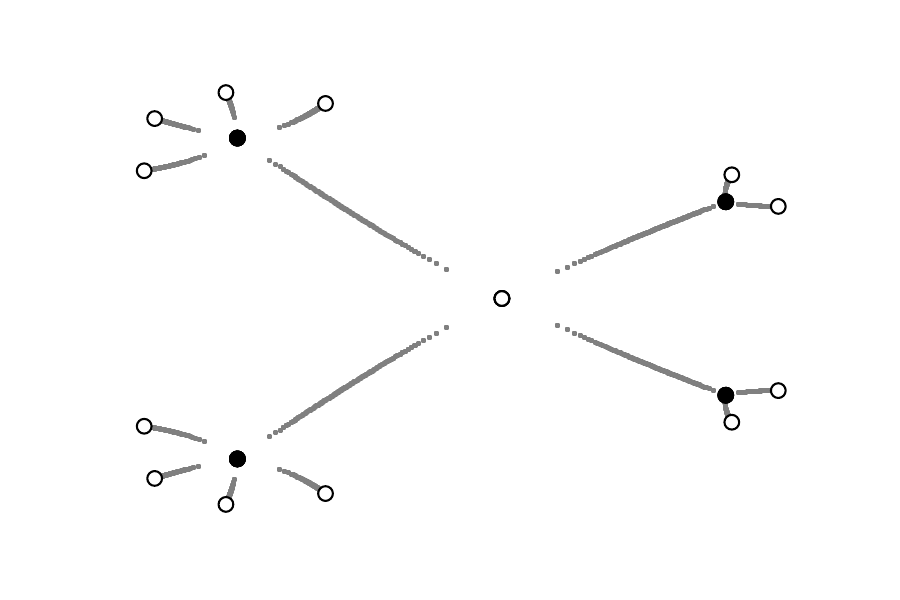}
    \end{minipage}
    \caption{Trees in \Fam{2} with $r=3$ and $s=5$.}
    \label{F:family2}
\end{figure}

\subsection{Shabat polynomials}

\begin{lemma}\label{lem:p2_2}
The Shabat polynomials for the family \Fam{2}, which for fixed $r,s$ consists of two trees
$T_{2,1}$ and $T_{2,2}$,
can be obtained by substituting appropriate values for $a,b,c$ in:
\begin{equation}\label{eq:ShabatP2}
    P(x) = (x^2+cx+a)^r(x^2-cx+b)^s.
\end{equation}
For $T_{2,1}$ we set:
\begin{align*}
    a &= -\frac{br}{s}, \\ 
    b &\in \mathbb{Q} \text{ can chosen arbitrarily, and } \\
    c &= 0.
\end{align*}
For $T_{2,2}$:
\begin{align*}
    a &= \frac{1}{108}\frac{c^2 (32r^3+75r^2s+78rs^2+31s^3)}{(r+s)^2 s}, \\ 
    b &= \frac{1}{108}{c^2(31r^3+75rs^2+78r^2s+32s^3)}{(r+s)^2 r} \\
    c &\in \mathbb{Q} \text{ can be chosen arbitrarily}.
\end{align*}
\end{lemma}
\begin{proof}
There are two black vertices with degree $r$, and two black vertices with degree $s$. Let us suppose that the two black vertices of the same degree are linked by a degree $2$ polynomial. Thus, the Shabat polynomial $P(x)$ 
looks like 
\begin{equation}
    \label{E:ShapeShabatPolyO2}
 P(x) = (x^2+cx+a)^r(x^2-cx+b)^s 
 \end{equation}
for some apropriate $a,b,c$.
On the other hand, the position of the white vertices indicate that $P$ 
must be
\begin{equation}\label{eq:whiteFactorisation}
    P(x) = (x-w_0)^4 \prod\limits_{i=1}^{2r+2s-4} (x-w_i) + 1
\end{equation} 
To compute the appropriate values of $a,b$ and $c$, we use the standard ``differentiation trick'' (frequently attributed to Atkin–Swinnerton-Dyer \cite{atkin1971modular}, see also \cite{sijsling2014computing}). 

The idea is to compute the derivative of $P(x)$ in two ways, using the two expressions for $P$ given by \cref{E:ShapeShabatPolyO2,eq:whiteFactorisation}. By unique factorization we may then equate (up to some constant) factors
of the same degree in each resulting factorisation of $P'(x)$ to obtain equations for our unknowns. 

Concretely, computing $P'(x)$ using \cref{eq:whiteFactorisation} we see that the $(x-w_0)^4$ factor in $P(x)$ yields the sole degree $3$ factor \begin{equation}\label{eq:oneWayP2}
    (x-w_0)^3
\end{equation} 
in $P'(x)$. 
Computing $P'(x)$ again, using \cref{E:ShapeShabatPolyO2} this time, we obtain the following degree $3$ factor:
\begin{equation}\label{eq:otherWayP2}
    (-2r-2s)x^3+(cr-cs)x^2+(c^2r+c^2s-2as-2br)x+acs-bcr
\end{equation} 
Since \cref{eq:oneWayP2,eq:otherWayP2} must equal up to a multiplicative constant we obtain a set of equations on $a, b, c$ and $w_0$. Solving these equations yields the two distinct sets of values for $a,b,c$ as desired. 
\end{proof}

The roots of $(x^2+cx+a)$ and $(x^2-cx+b)$ which appear \cref{eq:ShabatP2} 
correspond to the positions of the four black vertices of the trees in $\Fam{2}$ and their behaviour as a function of $c$ informs us of the
form $T_{2,1}$ and $T_{2,2}$ take when embedded on the plane: 
\begin{itemize}
    \item For $c = 0$, one pair of roots must be purely rational, the other purely imaginary (cf. the left tree in \cref{F:family2}).
    \item For $c \neq 0$ both pairs of roots are composed of complex conjugates (cf. the right tree in \cref{F:family2}).
\end{itemize}
In both cases complex conjugation exchanges vertices of the same degree,
realising geometrically the symmetries of the two trees.


Alternatively, we can obtain a Shabat polynomial for $T_{2,1}$ via a composition.
\begin{lemma}\label{lem:p2alt}
The Shabat polynomial $P$ for the family $T_{2,1}$
be represented as:
\begin{align*}
    P &= R \circ Q \\ 
    R &= (-1)^r \left(\frac{r}{s}\right)^s (x-1)^r \left(x + \frac{s}{r}\right)^s \\ 
    Q &= x^2.
\end{align*}

\begin{proof}
The polynomial $R$ can easily be verified to be Shabat, with critical points
$0,1,-\frac{s}{r}$ and critical values $y_0 = 0, y_1 = 1$. To prove the lemma,
it suffices to show that $R$ corresponds to the tree of \cref{fig:rstree},
since composing that tree with the star tree corresponding to $Q(x)=x^2$ yields
the desired result.

The two roots of $R$, which correspond to two black vertices, are $1$ and
$-\frac{s}{r}$ of multiplicities $r$ and $s$ respectively.

Moreover, we have that $R(0) = 1$, so that there is a white vertex at $0$, and $R'(0) = 0$, $R''(0) = \frac{r(r+s)}{s}$. Therefore, $0$ is a root of multiplicity two in $P(x)-1$ and the white vertex located at the origin is of degree $2$. The only tree with two black vertices
of degrees $r$ and $s$ and a single white vertex of degree $2$ is the one shown in \cref{fig:rstree}.
\end{proof}
\end{lemma}

\begin{figure}
    \centering
    \includegraphics[width=0.3\linewidth]{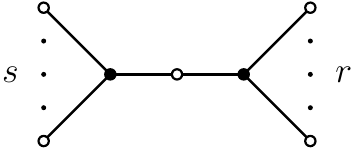}      
    \caption{The tree corresponding to $R$ in \cref{lem:p2alt}.}
    \label{fig:rstree}
\end{figure}

\subsection{Monodromy groups}

\begin{lemma}
Let $G$ be the monodromy group that corresponds to the tree $T_{2,1}$ in \Fam{2} with $(r,s)=d$ and $r=r_1d$, $s=s_1d$. Then 
\begin{small}
$$G=\left\{\begin{array}{ll} 
S_{4}^d \rtimes \Z_{2d} & \mathrm{if}\ (r_1,s_1)=(1,2)\\ 
(Q_8\rtimes S_{5})^d \rtimes \Z_d & \mathrm{if}\ (r_1,s_1)=(1,3)\\
((\Z_2^{r_1+s_1-1}\rtimes A_{r_1+s_1})\rtimes \Z_2)^d \rtimes \Z_d & \mathrm{if}\ r_1+s_1 \equiv 0 \Mod{2} 
, (r_1,s_1)\neq (1,3)\\
(\Z_2^{r_1+s_1-1}\rtimes S_{r_1+s_1})^d\times \Z_4 & \mathrm{if}\  r_1+s_1\equiv 1 \Mod{2}, 
(r_1,s_1)\neq (1,2)\\
\end{array}\right.$$
\end{small}
where $Q_8$ is the quaternion group of order 8.

\begin{proof}
By choosing the edge labels of $T_{2,1}$ appropriately we may have
$$\s_0=(1,\ldots,r)(r+1,\ldots,r+s)(r+s+1,\ldots,2r+s)(2r+s+1,\ldots, n)$$
$$\s_1=(1,r+1,r+s+1,2r+s+1)$$ and so
$$\s_{\infty}=\s_0\s_1=(1,2,\ldots,n)$$
Let $N$ be the subgroup of $G$ that is generated by $N=\<\s_{\infty}^i\s_1\s_{\infty}^{-i}, i=1,\ldots, r+s\>$. Then $N$ is a normal subgroup of $G$, generated by the elements 
\begin{small}
$$(r,s,2r+s,2r+2s),(r-1,r+s-1,2r+s-1,2r+2s-1), \ldots, (1,s+1,r+s+1,r+2s+1),$$
$$(s,r+s,r+2s,2r+2s), ((s-1,r+s-1,r+2s-1,2r+2s-1), \ldots, (r+1,r+s+1, 2r+s+1,1).$$
\end{small}

Notice that the generators of $N$ partition into $d$ different disjoint subsets and so $N$ consists of $d$ isomorphic copies of a subgroup $N_1$ which we may choose to be the subgroup generated by
\begin{small}
$$(r,r+s,2r+s,2r+2s), (r-d,r+s-d,2r+s-d,2r+2s-d),\ldots,$$ $$(r-(r_1-1)d, r+s-(r_1-1)d, 2r+s-(r_1-1)d, 2r+2s-(r_1-1)d),$$
$$(s,r+s,r+2s,2r+2s),(s-d,r+s-d,r+2s-d,2r+2s-d), \ldots, $$ $$(s-(s_1-1)d,r+s-(s_1-1)d,r+2s-(s_1-1)d,2r+2s-(s_1-1)d).$$
\end{small}
If $(r/d,s/d)=(1,2)$ then $r=d$, $s=2d$ and the generators of $N_1$ are $$(d,3d,4d,6d),(2d,3d,5d,6d),(d,2d,4d,5d).$$ The subgroup generated by these elements is isomorphic to the subgroup generated by $$\<(1,3,4,6),(2,3,5,6),(1,2,4,5)\>\cong S_4.$$ Moreover, if $(r_1,s_1)=(1,2)$ then the complement classes representatives of $N$ in $G$ generate $\Z_{2d}$ and so the first part of the Lemma is proved.

If $(r_1,s_1)=(1,3)$ then $r=d$, $s=3d$ and the generators of $N_1$ are $$(d,4d,5d,8d),(3d,4d,7d,8d),(2d,3d,6d,7d),(d,2d,5d,6d)$$ and the subgroup generated by these elements is isomorphic to $$\<(1,4,5,8),(3,4,7,8),(2,3,6,7),(1,2,5,6)\>\cong Q_8\rtimes S_4.$$ Moreover, the complement classes representatives of $N$ in $G$ generate $\Z_d$ and that implies the second part of the Lemma.

Assume now that $(r_1,s_1)=(r/d,s/d)\not\in\{(1,2),(1,3)\}$. Then we may assume that $N_1$ is generated by the elements 
$$a_0=(1,\ldots, r_1)(r_1+1, \ldots r_1+s_1)(r_1+s_1+1,\ldots, 2r_1+s_1)(2r_1+s_1+1,\ldots 2r_1+2s_1)$$
$$a_1=(1,r_1+1,r_1+s_1+1,2r_1+s_1+1)$$ and so
$$a_{\infty}=a_0a_1=(1,\ldots,n)$$

If $r_1+s_1\equiv 1\pmod2$ then in $N_1$ the subgroup $M=\<a_{\infty}^ia_1^2a_{\infty}^{-i}\>$ is normal in $N_1$ and isomorphic to $\Z_2^{r_1+s_1-1}$. Moreover $(N_1/M)\cong S_{r_1+s_1+1}$. 

Finally, let $r_1+s_1\equiv 0\pmod2$. Assume that $g=[a_0,a_1]$ and take the normal subgroup $M=\<a_{\infty}^ig a_{\infty}^{-i}, i=1,\ldots, r_1+s_1\>$. $M$ has index 2 and $N1\cong M\rtimes \Z_2$. On the other hand, the subgroup $M_1=\< a_{\infty}^i a_1^2 a_{\infty}^{-i}\>\cong \Z_2^{r_1+s_1-1}$ is normal in $M$ and the conjugacy classes representatives of $M_1$ in $M$ consists of the elements $(1,\ldots, r_1+s_1-2,r_1+s_1)(r_1+s_1+1,\ldots,2r_1+2s_1-2,2r_1+2s_1), (1,\ldots,r_1+s_1-3,r_1+s_1-1,r_1+s_1)(r_1+s_1+1,\ldots, 2r_1+2s_1-3, 2r_1+2s_1-1,2r_1+2s_1)$ which generate a subgroup isomporphic to $A_{r_1+s_1}$ if $r_1+s_1\equiv 0\pmod2$.
\end{proof}
\end{lemma}

\begin{lemma}
The monodromy group $G$ that corresponds to tree $T_{2,2}$ in \Fam{2} is
$$G=\left\{\begin{array}{ll} S_5 & \mathrm{if}\ r=1\ \mathrm{and}\ s=2\\ 
S_n & \mathrm{if}\  d=(r,s)=1\ \mathrm{and}\ s\neq 2\\
S_{n/d}^d\rtimes \Z_d & \mathrm{if}\ d=(r,s)\neq 1\  \mathrm{and}\  s\neq 2r\\
S_{(n/d)-1}^d\rtimes \Z_d & \mathrm{if}\ d=(r,s)\neq 1\  \mathrm{and}\  s=2r
\end{array}\right.$$ where $n=2r+2s$ and $r < s$.

\begin{proof}
One can easily see that we can choose the tree labels such that the two generators of the monodromy group are $$\s_0=(1,2,\ldots, r)(r+1,\ldots, 2r)(2r+1, \ldots, 2r+s)(2r+s+1, \ldots, n)$$ and $$\s_1=(1,r+1,2r+1,2r+s+1).$$ Then $$\s_{\infty}=\s_0\s_1=(1,2,3,\ldots, n).$$ Now let us take again the orbit of $\s_1$ under the conjugate action of $\s_{\infty}$. Namely, let us generate the elements $g_k=\s_{\infty}^{-k}\s_1\s_{\infty}^k$ for all $k=1,\ldots, n$. 

It is obvious that $N=\<g_k, k=1,\ldots, n\>$ is a normal subgroup of $G$. One can easily see that if $(r,s)=1$ the elements $g_k$ generate a normal subgroup of $S_n$ which contains an odd permutation, hence they generate the entire group $S_n$.

On the other hand, if $(r,s)=d\neq 1$ then the 4-cycles $g_k$ create a partition of $\{1,\ldots n\}$ and each subset of the partition contains $n/d$ elements. So we can partition the elements $g_k$ in the following subsets $$D_i=\{s_{i,k}=(i+kd,r+1+i+kd,2r+1+i+kd,2r+s+1+i+kd), k=0,\ldots (n/d)-1 \}$$ for $i=1,\ldots, r-1$.  It is obvious that the elements of $\<D_i\>$ commute with the elemets of $\<D_j\>$ for every $i\neq j$. Moreover, if $s\neq 2r$ then each $\<D_i\>$ is normal in $S_{n/d}$ and contains an odd permutation so it is isomorphic to $S_{n/d}$ and hence $\<g_k, k=1,\ldots, n\>\cong S_{n/d}^d$.      
    On the other hand, if $s=2r$ then the product of the elements of $D_i$, $s_{i,1}s_{i,2}\ldots s_{i,k}=1$, the trivial cycle, and so $\<D_i\>\cong S_{(n/d)-1}$, hence $N=\<g_k, k=1,\ldots, n\>\cong S_{(n/d)-1}^d$.
    
Finally, since $N$ is a normal subgroup of $G$ that admits a complement to $G$ and $G/N\cong \Z_d$, the result follows.
    
\end{proof}
\end{lemma}

\section{Family \texorpdfstring{\Fam{3}}{O_3}}

The criterion for the family \Fam{3} to split into two Galois fixed points
is this time not combinatorial but ``Diophantine'', as remarked in \cite{lando2013graphs}.
In more detail, for certain values of $r,s$ (such as $r=5,s=6$) the discriminant
of the corresponding defining polynomial is a perfect square, meaning that the two
Shabat polynomials are defined over $\mathbb{Q}$. For all other values of $r,s$
they are defined over some quadratic field. 
Figure \ref{F:family3} shows the two distinct trees in the case $r = 3$ and $s=5$.

\begin{figure}[ht]
    \centering
    \begin{minipage}[b]{0.48\textwidth}
        \includegraphics[width=\linewidth]{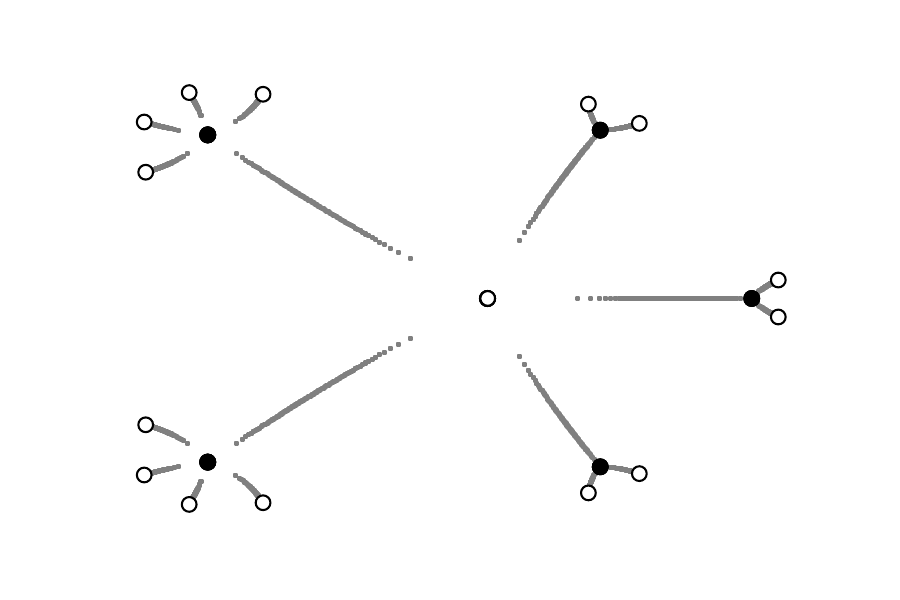}
    \end{minipage}
    \hfill
    \begin{minipage}[b]{0.48\textwidth}
        \includegraphics[width=\linewidth]{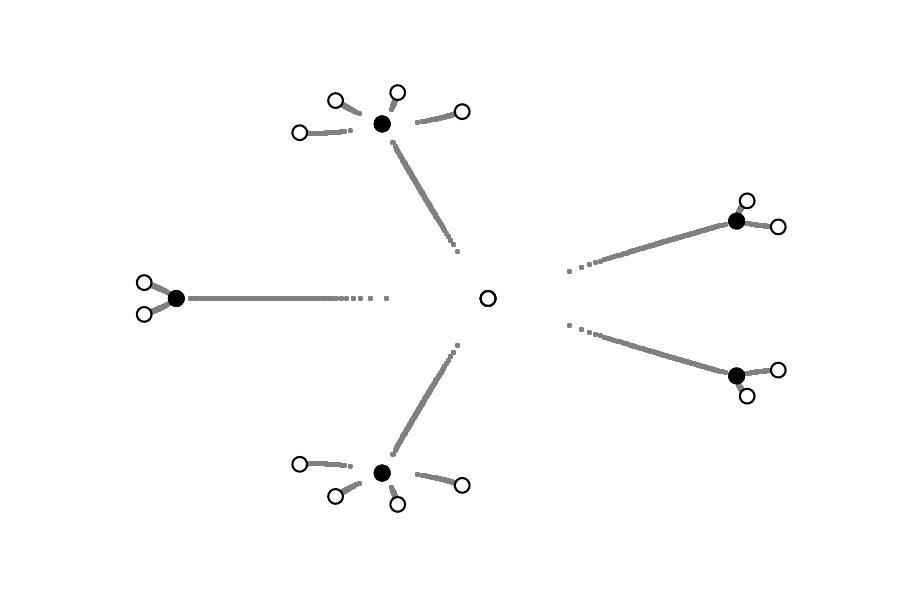}
    \end{minipage}
    \caption{Trees for \Fam{3} with $r=3$ and $s=5$.}
    \label{F:family3}
\end{figure}

\subsection{Shabat polynomials}

\begin{lemma}\label{lem:p3}
The Shabat polynomials for the combinatorial family \Fam{3} are:
\begin{align*}
    P(x) &= (x^3+x^2+cx+a)^r(x^2+x+b)^s, \\
    a &= \frac{-7s^3+(96d+r)s^2+(48dr+51r^2)s-144r^2d+51r^3}{96s(3r+2s)^2} \\ 
    b &= \frac{27r^2+(-24d+34s)r-48sd+11s^2}{72r^2+48rs} \\ 
    c &= 3d/(3r+2s)
\end{align*}
Where $d$ is a root of the defining polynomial
\begin{align}\label{eq:defPolyF3}
    D(x) &= (6912r^2+18432rs+9216s^2)x^2 \\  
         &+ (-4320r^3-13824r^2s-12768rs^2-3648s^3)x \\ 
         &+ 651r^4+2460r^3s+3210r^2s^2+1772rs^3+355s^4
\end{align}
\end{lemma}
\begin{proof}
Assume that the $3$ degree $r$ black vertices are linked by a degree $3$ polynomial, and that the $2$ degree $s$ black vertices are linked by a degree $2$ polynomial, so that on the one hand $P$ has shape
\begin{equation}
    \label{E:ShapeShabatPolyO3}
P(x) = (x^3+x^2+cx+a)^r(x^2+x+b)^s.
\end{equation}
On the other hand, for the white vertices, $P$ has shape
\[ P(x) = (x-w_0)^5 \prod\limits_{i=1}^{3r+2s-5} (x-w_i) + 1. \]
Using the differentiation trick, we obtain two degree $4$ polynomials below, with the appropriate constants~:
\[ (3r+2s)(x-w_0)^4 \]
and 
\[ (3r+2s)x^4+(5r+3s)x^3+(3br+cr+2cs+2r+s)x^2+(2as+2br+cr+cs)x+bcr+as.
\]
Making all the coefficients equal gives a set of $4$ equations on $a$,$b$,$c$ and $w_0$ that we can solve to obtain two solutions, depending on two roots of the polynomial $D(x)$ above.
\end{proof}

\subsection{Monodromy groups}

\begin{lemma}
The monodromy group $G$ that corresponds to the family of trees \Fam{3} is
$$G=\begin{cases}
(A_{n/d})^d\rtimes \Z_{2d} \text{  if $n/d$ is even}\\
(A_{n/d})^d\rtimes \Z_d \text{  if $n/d$ is odd}
\end{cases}$$
where $d=\gcd(r,s)$ and $n=3r+2s$.

\begin{proof}
The edge labels can be chosen in such a way that
$$\s_0=(1,\ldots, r)(r+1,\ldots, 2r)(2r+1,\ldots, 3r)(3r+1,\ldots, 3r+s)(3r+s+1,\ldots, n)$$
$$\s_1=(1, r+1, 2r+1, 3r+1, 3r+s+1)$$
in the first tree and
$$\s_0=(1,\ldots, r)(r+1, \ldots, 2r)(2r+1,\ldots, 2r+s)(2r+s+1, \ldots, 2r+2s+1)(2r+2s+1,\ldots, n)$$
$$\s_1=(1, \ldots, r+1, 2r+1, 2r+s+1, 2r+2s+1)$$ in the second case. In both cases
$\s_{\infty}=\s_0\s_1=(1,2,\ldots, n)$

Again, the orbit of $\s_1$ under the action of $\s_{\infty}$ is  a normal subgroup, partitioned into $n/d$ disjoint subsets each one generating a copy of $A_{n/d}$. The result follows easily by examining the quotient $G/N$ and the fact that $N$ admits a complement in $G$. 
\end{proof}
\end{lemma}

\section{Family \texorpdfstring{\Fam{4}}{F_4}}

The family \Fam{4} also forms a Galois 2-orbit,
with the trees being defined over $\mathbb{Q}(\sqrt{-3})$.
Figure \ref{F:family4} shows the two distinct trees in the case $r = 4$ and $s=5$.

\begin{figure}[ht]
    \centering
    \begin{minipage}[b]{0.48\textwidth}
        \includegraphics[width=\linewidth]{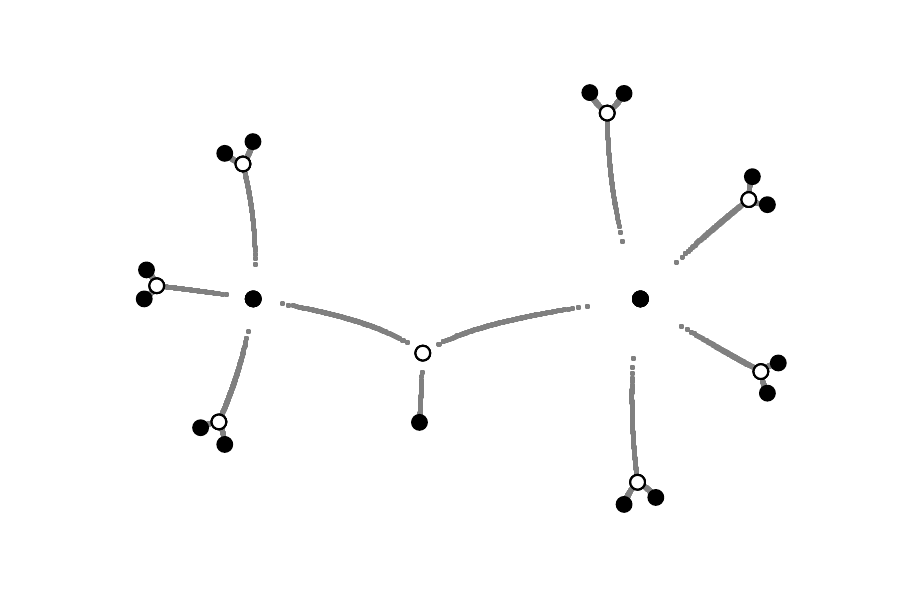}
    \end{minipage}
    \hfill
    \begin{minipage}[b]{0.48\textwidth}
        \includegraphics[width=\linewidth]{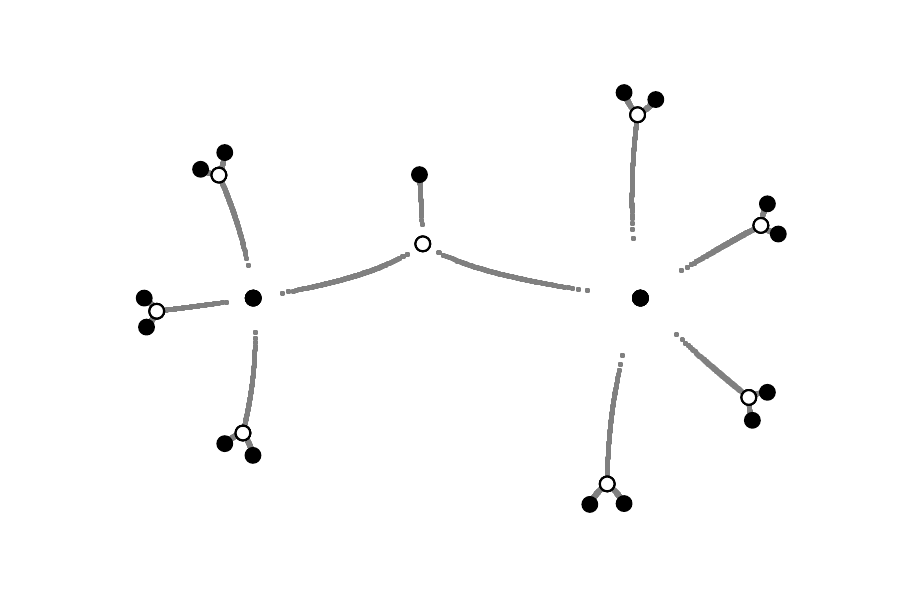}
    \end{minipage}
    \caption{Trees for \Fam{4} with $r=4$ and $s=5$.}
    \label{F:family4}
\end{figure}

\subsection{Shabat polynomials}

In computing the Shabat polynomials for \Fam{4} we'll come across
the dessins known as $(r,s)$-brushes, as introduced in \cite{adrianov2020davenport}.
These brushes and their polynomials play a significant role
in our computations for \Fam{4} as well as
\Fam{5} and \Fam{6} and so we collected some results on them in \cref{SS:ShabatForBrush}, which we invite the reader to consult.

\begin{lemma}\label{lem:p4}
The Shabat polynomials for the combinatorial family \Fam{4} are given by the composition $Q(P(x))$ where
\begin{align}
    P(x) &= \left(\frac{x+1}{2}\right)^{r} J_{s-1}(-s,r,x), \\ 
    Q(x) &= \frac{3\alpha}{2}x(x+1)(2x+1+\alpha),
\end{align}
where $\alpha$ is a root of the defining polynomial $x^2+3$ and $J_n(x,a,b)$ is
the $n$-th Jacobi polynomial with parameters $a,b$. 
\end{lemma}
\begin{proof}
The two trees are compositions of an $(r-1,s-1)$-brush, as depicted in \cref{fig:rstree2} 
with the 3-star with a white center whose black vertices are positioned at $0$, $-1$ and $\frac{-1 - \alpha}{2}$ where $\alpha = \pm \sqrt{-3}$. The polynomial $Q$ is multiplied by the constant $\frac{3\alpha}{2}$ to make it a Shabat polynomial having $1$ as a critical value. 
Figure \ref{F:compF4} show the elements of the composition.

\begin{figure}[ht]
    \centering
        \includegraphics[scale=1.5]{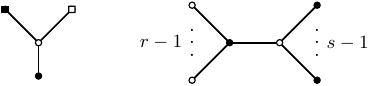}
    \caption{Trees involved in the composition for for \Fam{4}.}
    \label{F:compF4}
\end{figure}
\end{proof}

\subsection{Monodromy groups}

\begin{lemma}
The monodromy group $G$ that corresponds to the family of trees \Fam{4} is
$$G=\begin{cases}
A_p^3 \rtimes A_4 \text{  if $r\equiv s \equiv 0\pmod 2$}\\
A_p^3 \rtimes \Z_3 \text{  if $r\equiv s\equiv 1\pmod 2$}\\
A_p^3 \rtimes (A_4\times \Z_2) \text{  if $r\not \equiv s\pmod 2$} 
\end{cases}$$

\begin{proof}
Let $\s_0$ be the generator induced by the black vertices and $\s_1$ the generator induced by the white vertices.
Then, we choose the labels on the tree in the following way
$$\s_0=(1,2,3)(4,5,6)\ldots (3r-5,3r-4,3(r-1))(3r-2,n-1,n)(3r-1,3r,3r+1)\ldots (n-4,n-3,n-2)$$ 
$$\s_1=(1,4,\ldots,3r-2)(3r-1,3r+4,\ldots, n-1)$$
for the first choice of the tree and 
$$\s_0=(1,2,3)(4,5,6)\ldots (3r-5,3r-4,3(r-1))(3r-2,3r-1,n)(3r,3r+1,3r+2)\ldots (n-3,n-2,n-1)$$ 
$$\s_1=(1,4,\ldots, 3r-2)(3r,3r+3,\ldots,n)$$
for the second choice of the tree.
The reason for this choice of labelling is to make sure that in both cases we have $\s_{\infty}=\s_0\s_1=(1,2,\ldots,n)$.

By assumptions $r<s$, and so it suffices to take the normal closure of $\s_1^r$ under the conjugate action of $\s_{\infty}$. Then, one can easily check that the elements
$\s_{\infty}^i\s_1^r\s_{\infty}^{-i}$ for $i=1,\ldots, n$ are partioned into three disjoint orbits, each one of them generating a copy of $A_p$.   
Hence, in both cases, the normal subgroup $$H=\<\s_{\infty}^i\s_1^r\s_{\infty}^{-i}, i=1,\ldots, n\>\cong A_p^3.$$  
If $r\equiv s\equiv 0 \pmod 2$ then $H$ is of index 12 in $G$ and its complement in $G$ is always isomorphic to $A_4$.
If $r\equiv s\equiv 1\pmod 2$ then $H$ has index 3 in $G$ and so its complement is isomorphic to $\Z_3$.
If $r\not\equiv s\pmod 2$ then $H$ is of index 24 in $G$ and its complement is isomorphic to $A_4\times\Z_2$. 

Notice that one may take the conjugate action of $\s_{\infty}$ on $\s_1^s$. In that case, we get again three different orbits and $H\cong S_p^3$. 

\end{proof}
\end{lemma}

\section{Family \texorpdfstring{\Fam{5}}{F_5}}

The family \Fam{5} consists of two trees obtained via appropriate compositions. The two trees are define over $\mathbb{Q}$ and so
form two Galois fixed points.
Both trees have a degree-$4$ white vertex adjacent to two unique black vertices of degree $r$ and two black leafs, let's call it 
the {\em center} of the trees. By considering the cyclic arrangement of 
the four neighbours of the center we distinguish two cases: the
two black leaves may be consecutive or not.
See for example \ref{F:family5} in which we present the two distinct trees in the case $r = 4$.

\begin{figure}[ht]
    \centering
    \begin{minipage}[b]{0.48\textwidth}
        \includegraphics[width=\linewidth]{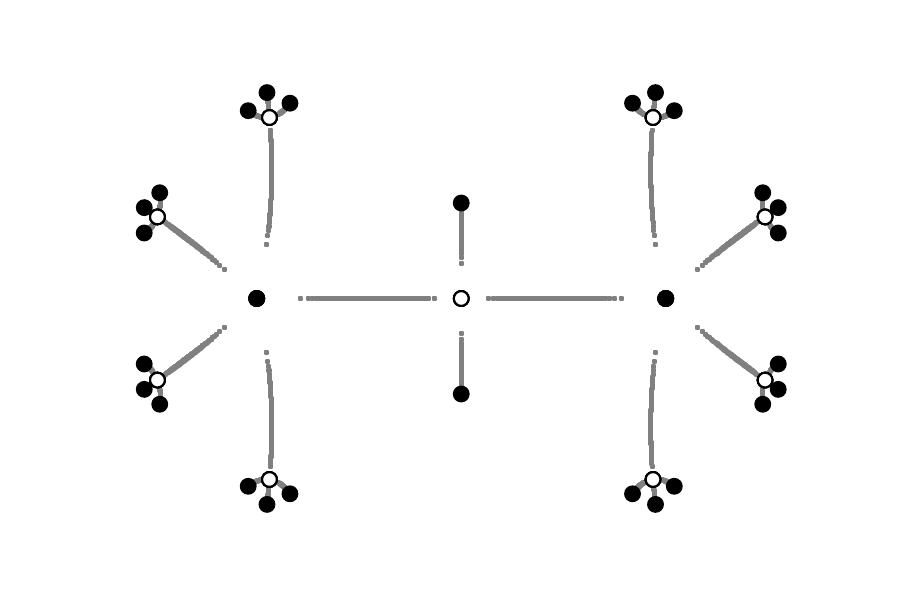}
    \end{minipage}
    \hfill
    \begin{minipage}[b]{0.48\textwidth}
        \includegraphics[width=\linewidth]{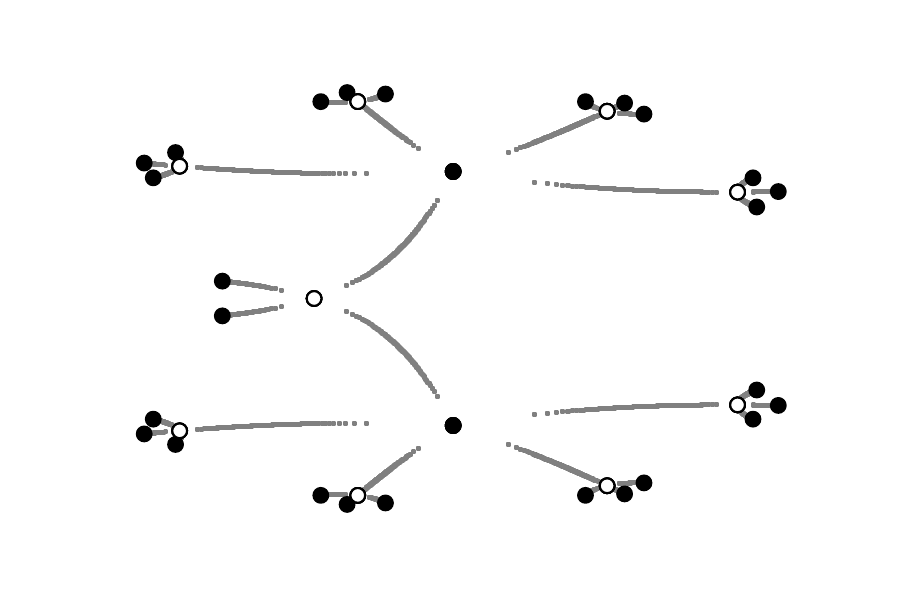}
    \end{minipage}
    \caption{Trees for \Fam{5} with $r=4$.}
    \label{F:family5}
\end{figure}

\subsection{Shabat polynomial}
\label{SS:ShabatO5}

\begin{lemma}\label{lemma:o5}
The Shabat polynomials for the the two trees of the family \Fam{5} are given by the compositions
$Q_1(P(x))$ and $Q_2(-P(ix))$, where $P(x)$ is
\begin{equation*}
P(x) = K \sum\limits_{k=0}^{r-1} (-1)^k \binom{r-1}{k} \frac{x^{2r-2k-1}}{2r-2k-1} + C
\end{equation*}
with
\begin{equation*}
 C = - \frac{1}{2}, \quad K = \frac{(-1)^{r+1} (2r-1)!!}{2(2(r-1))!!},
\end{equation*}
and
\begin{align*}
    Q_1(x) &= -(2x+1)^4+1, \\
    Q_2(x) &= 4x(x-1)(x-i)(x-1-i).
\end{align*}
\end{lemma}

\begin{proof}

Both trees of in \Fam{5} are a composition of an $(r-1,r-1)$-brush and a 4-star with a white center. The polynomial $P(x)$ for the $(r-1,r-1)$-brush
is computed in \cref{SS:ShabatForBrush}.

Recall that to compose two dessins, say $A \circ B$, we must mark two vertices of $B$. In the case of the 4-star there's two ways to mark black vertices: the
two black vertices can be adjacent or non-adjacent with respect to the cyclic order around the 
root of $Q$. These two ways of marking and then composing yield the two trees 
of \Fam{5}, corresponding to the two possible arrangements of leaves around their center.

Let us now compute the polynomials $Q_1(x)$ and $Q_2(x)$ corresponding to
the marked dessins. The first one, $Q_1(x)$, is readily verified to be a polynomial with $1$ as unique critical value corresponding to the $4$-star with a white center and two non-consecutive black vertices at $0$ and $-1$. 
Its composition with the $(r-1,r-1)$-brush having $-1$ as a critical value gives the tree in \Fam{5} in which the two black leaf vertices adjacent
to the center are not consecutive.

To obtain the other tree, we rotate the brush in the plane, so that its critical values become $0$ and $1$. This has a geometric explanation~: in the tree with two consecutive black leafs adjacent to the center, all vertices of the same degree and color are in complex-conjugate positions. Hence, we consider $A(x) = -P(ix)$ which is a Shabat polynomial for the $(r-1,r-1)$-brush with $0$ and $1$ as critical values. We next compute
the polynomial corresponding to the $4$-star with black vertices positioned at $0$, $1$, $i$ and $1+i$ which one can easily verify is $Q_2(x)$.

Finally, we note that the resulting compositions yield rational polynomials.
This is easy to see in the case of $Q_1(P(x))$ but less so in the
case $Q_2(-P(ix))$. To see why, we begin by noting that the constant term 
of $P(x)$ is $- \frac{1}{2}$. Therefore, the constant term in $A(x)$ is $\frac{1}{2}$ and $A(x)$ can be written as \[ A(x) = a(x) + \frac{1}{2} \]
where $a(x)$ has no constant term and its monomials are of the form $ix^k$ where $k$ is odd. Plugging this into $Q(x)$, we obtain
\[ 4a(x)^4 -6a(x)^2-8i a(x)^3 + 2i a(x) + \frac{5}{4} \]
in which all of the terms are more easily seen to be rational.
\end{proof}

\subsection{Monodromy group}
\begin{lemma}
The monodromy group $G$ that corresponds to the family of trees \Fam{5} is, 
$$G=\begin{cases}
A_p^2 \rtimes \Z_4 \text{  if $r\equiv 1\pmod 2$}\\
S_p^2 \rtimes \Z_4 \text{  if $r\equiv 0\pmod 2$}\\ 
\end{cases}$$
for the first case or 
$$G=\begin{cases}
A_p^4\rtimes (\Z_2^3\rtimes Z_4) \text{. if $r\equiv 0\pmod 2$}\\
A_p^4\rtimes \Z_4 \text{ if $r\equiv 1\pmod 2$}
\end{cases}$$
for the second tree as depicted in picture 2.
\begin{proof}
By appropriate choice of labels we have that the generators of the groups are
$$\s_0=(1,5,\ldots,4r-3)(4r-1,4r+1,4r+5,\ldots,8(r-1)+1)$$ for 
the first choice of the tree and 
$$\s_0=(1,5,\ldots,4r-3)(4r-2,4r+1,4r+5,\ldots, 8(r-1)+1)$$ for the second choice of the tree. In both cases
$$\s_1=(1,2,3,4)(5,6,7,8)\ldots (8(r-1)+1 ,\ldots , 8(r-1)+4)$$

Let us restrict to the first case and let $H$ be the subgroup generated by the conjugating action of $\s_0$ to $\s_1$, namely 
$H=\< \s_1^i\s_0\s_1^{-i}, i=0,\ldots,3\>$. The orbit is partioned into two subsets, giving either two copies of $A_p$ or two copies of $S_p$ depending on whether $r\equiv 1\pmod 2$ or $r\equiv 0\pmod 2$. The result follows since $\s_1$ acts on $H$ as an automorphism. 

For the second case, if $r\equiv 1\pmod2$ then again the orbit of $\s_0$ under the action of $\s_1$ generates $A_p^4$ and the result follows.
If $r\equiv 0\pmod2$ let $s=[\s_0,\s_1]$ and $m_i=\s_0^is\s_0^{-i}$, $i=0,1,2,3$. The subgroup $N$ generated by $N=\<(m_im_{i+1})^2, i=0,1,2,3\>$ where indices are considered $\pmod5$, is normal in $G$ and isomorphic to $A_p^4$. Moreover, the complement of $N$ is always of order 32 and it is isomorphic to $\Z_2^3 \rtimes \Z_4$.
\end{proof}
\end{lemma}

\section{Family \texorpdfstring{\Fam{6}}{F_6}}
The combinatorial family \Fam{6} forms a Galois 2-orbit defined over $\mathbb{Q}(\sqrt{5})$. Figure \ref{F:family6} shows the two distinct trees in the case $r = 3$.

\begin{figure}[ht]
    \centering
    \begin{minipage}[b]{0.48\textwidth}
        \includegraphics[width=\linewidth]{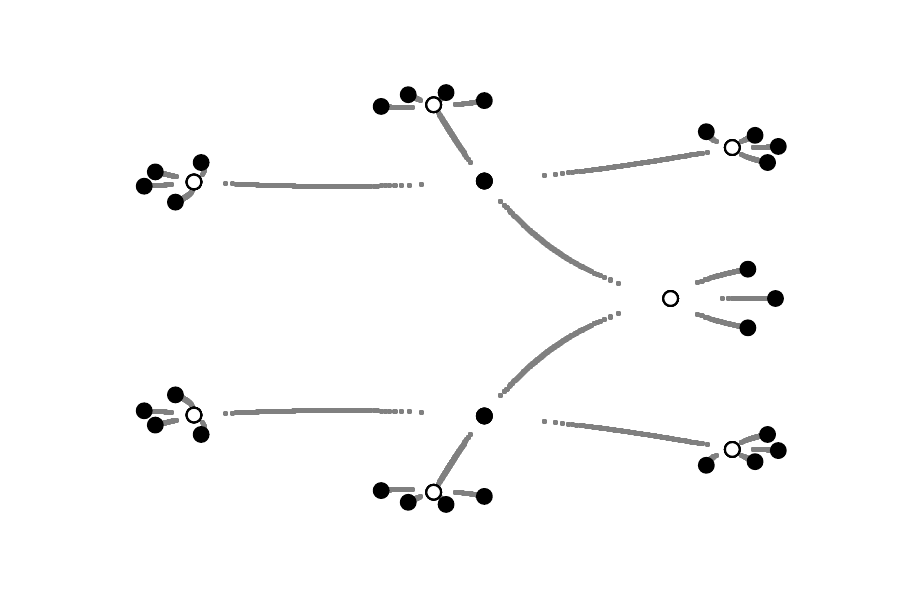}
    \end{minipage}
    \hfill
    \begin{minipage}[b]{0.48\textwidth}
        \includegraphics[width=\linewidth]{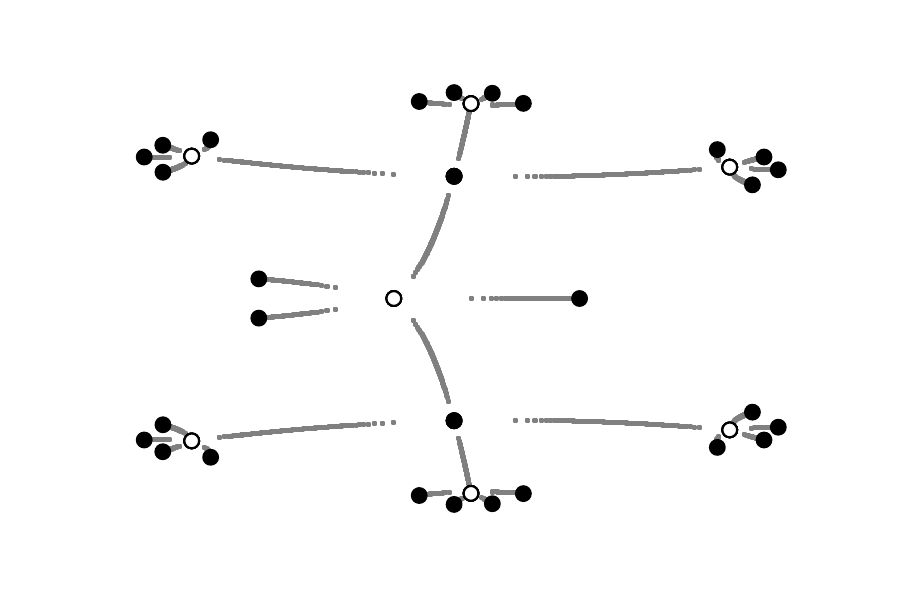}
    \end{minipage}
    \caption{Trees in \Fam{6} with $r=4$.}
    \label{F:family6}
\end{figure}

\subsection{Shabat polynomial}
\label{SS:ShabatO6}

\begin{lemma}\label{lemma:o6}
The Shabat polynomials for the two trees of the combinatorial family \Fam{6} are given
by the composition $Q(P(x))$ where

\begin{align*}
    P(x) &=  \frac{(2r-1)!}{((r-1)!)^2 2^{2r}} \cdot 
\sum\limits_{k=0}^{r-1}  \binom{r-1}{k} (10-2\alpha)^{-r+k+1} \frac{x^{2r-2k-1}}{2r-2k-1} - \frac{1}{4} \left(1 + \alpha \right), \\ 
    Q(x) &= 1-x^5,
\end{align*} 
where the constant $\alpha$ takes the value $\sqrt{5}$ or $-\sqrt{5}$ 
for each tree, respectively.
\end{lemma}
\begin{proof}
Notice that the two trees can be obtained by composing the $5$-star of white center with the $(r-1,r-1)$-brush as depicted in \cref{fig:rstree}. 
We show in \cref{SS:ShabatForBrush} that the polynomial for the brush with fifth roots of unity as critical values is
\[ \frac{(-1)^{r-1} (2r-1)!}{2^{2r} (r-1)!^2} \sqrt{-10+2\sqrt{5}} \cdot  
\sum\limits_{k=0}^{r-1} (-1)^k \binom{r-1}{k} \frac{x^{2r-2k-1}}{2r-2k-1} - \frac{1}{4} \sqrt{5} - \frac{1}{4}. \]
In order to get rid of the factor $M = \sqrt{-10 + 2 \sqrt{5}}$, we make a change of variable $x \mapsto \frac{x}{M}$. This gives
\[ 
P_1(x) = \frac{(2r-1)!}{2^{2r} (r-1)!^2} \cdot \sum\limits_{k=0}^{r-1} (-1)^{k+r-1} \binom{r}{k} M^{2k-2r-2} \frac{x^{2r-2k-1}}{2r-2k-1} - \frac{1}{4} \sqrt{5} - \frac{1}{4}
\] 
and we obtain the one of the two desired polynomials since $M^2 = -10 + 2 \sqrt{5}$ and $(-1)^{k+r-1} = (-1)^{k-r+1}$. Applying the involution $\sqrt{5} \mapsto -\sqrt{5}$ we obtain a new
Shabat polynomial $P_2$, not equivalent to $P_1$, which corresponds to the second 
tree in the family.
\end{proof}

\subsection{Monodromy groups}

\begin{lemma}
The monodromy group $G$ that corresponds to the family of trees \Fam{6} is
$$G=\begin{cases}
A_p^5 \rtimes \Z_5 \text{  if $r\equiv 1\pmod 2$}\\
A_p^5 \rtimes (\Z_2^4\rtimes \Z_5) \text{  if $r\equiv 0\pmod 2$}\\ 
\end{cases}$$

\begin{proof}
For the first tree, by choosing the edge labels appropriately we may have
$$\s_0=(1,\ldots, 5)(6,\ldots, 10)\ldots (5(p-1), \ldots, 5p)$$
$$\s_1=(1, 6, \ldots, 5(r-1)+1=5r-4)(5r-2, 5r+1, 5r+6, \ldots, n-4)$$
We define $s=[\s_0,\s_1]$ and let $H$ be the normal closure of $s$ under the conjugate action of $\s_0$, that is $H=\< m_i=\s_0^is\s_0^{-i}, i=0,\ldots,4\>$. Then $G/H\cong \Z_5$ and $G\cong H\rtimes \Z_5$.

Now if $r\equiv 1\pmod 2$ then $H\cong A_p^5$. 
and so $G\cong A_p^5\rtimes \Z_5$.

If $r\equiv 0\pmod 2$ then let $N=\<(m_im_{i+1})^2, i=1,\ldots, 5 \>$ where indices are considered $\pmod 5$. One can easily check that $N$ is normal in $G$ (and so in $H$) and that $N\cong A_p^5$. Moreover, $H/N\cong \Z_2^4$ and the result follows. 

For the second tree, we have that 
$$\s_0=(1,\ldots,5)(6,\ldots,10)\ldots (5(p-1),\ldots, 5p)$$
$$\s_1=(1,6,\ldots, 5(r-1)+1)(5r-3, 5r+1,\ldots, n-4)$$
Let again $H=\< m_i=\s_0^is\s_0^{-i}, i=0,\ldots,4\>$. Then $H$ is normal in $G$ and $G/H\cong \Z_5$. Now if $r\equiv 1\pmod 2$ then $H\cong A_5^p$ and so $G\cong A_5^p\rtimes \Z_5$.

If $r\equiv 0\pmod 2$ then take $N=\<\s_1m_i^2\s_1^{-1}, i=1,\ldots, 5\>$. Then $N$ is a normal subgroup of $G$ (and so of $H$) and $N\cong A_5^p$. Also $H/N\cong \Z_2^4$ and so the result follows.
\end{proof}
\end{lemma}

\section{Sporadic dessins}

\subsection{Shabat polynomials}
The Shabat polynomials of the families \Fam{7}, \Fam{8} are already known to thanks to \cite{betremaZvonkinList}. They are both defined over quadratic fields: \Fam{7} is defined over $\mathbb{Q}(\sqrt{-14})$ while \Fam{8} is defined over $\mathbb{Q}(\sqrt{21})$. These are an imaginary and a real quadratic field, reflecting the presence or lack of symmetry between the
respective pairs of trees. 

We continue with a calculation of the Shabat polynomials for \Fam{9}, 
defined over $\mathbb{Q}(\sqrt{-3})$.
\begin{lemma}\label{lemma:o9}
For the sporadic family \Fam{9} we have 
\begin{equation*}
\resizebox{0.95\textwidth}{!}{%
    $P(x) = (x-1)\left( x^4 + \frac{8}{7}x^3 + \frac{6}{7}x^2\left(-\frac{5}{7}
    + \frac{6\alpha}{7}\right) + \left(-\frac{40}{49} + \frac{48
    \alpha}{49}\right)x -  \frac{59}{2401} - \frac{156\alpha}{2401}\right)^{2}$%
    }
\end{equation*}
where $\alpha = \pm \sqrt{-3}$.
\end{lemma}
\begin{proof}
We apply the differentiation trick. Assume that the positions of the four degree-$2$ black vertices are roots of a degree-$4$ polynomial, so that $P(x)$ is of the form 
\begin{equation}\label{eq:oneWayF9}
    P(x) = (x^4 + b_0 x^3 + b_1 x^2 + b_2 x + b_3)^2 (x-b_4).
\end{equation}
On the other hand, if we assume that the two white vertices of degree $3$ are related, and denote by $w_1,w_2$ and $w_3$ the positions of the remaining degree-1 white vertices, then
\begin{equation}\label{eq:otherWayF9}
    P(x) = (x^2-w_0)^3 (x-w_1)(x-w_2)(x-w_3)+1.
\end{equation}

Taking the derivatives of \cref{eq:oneWayF9,eq:otherWayF9},
we obtain two degree-$4$ factors from each expression which
we choose to pair up to form two sets of equations. In more
detail, we set
\begin{equation*}
    9\,{z}^{4}+7\,b_{{0}}{z}^{3}-8\,{z}^{3}b_{{4}}-6\,{z}^{2}b_{{0}}b_{{4}
    }+5\,b_{{1}}{z}^{2}-4\,zb_{{1}}b_{{4}}+3\,b_{{2}}z-2\,b_{{2}}b_{{4}}+b
    _{{3}}
\end{equation*}
equal to  
\begin{equation*}
\resizebox{0.95\textwidth}{!}{%
$
\begin{split}
  &9\,{z}^{4}-8\,{z}^{3}w_{{1}}-8\,{z}^{3}w_{{2}}-8\,{z}^{3}w_{{3}}+7\,{z
  }^{2}w_{{1}}w_{{2}}+7\,{z}^{2}w_{{1}}w_{{3}}+7\,{z}^{2}w_{{2}}w_{{3}}-
  6\,zw_{{1}}w_{{2}}w_{{3}} \\ 
  & -3\,{z}^{2}w_{{0}} +2\,zw_{{0}}w_{{1}}+2\,zw_{
  {0}}w_{{2}}+2\,zw_{{0}}w_{{3}}-w_{{0}}w_{{1}}w_{{2}}-w_{{0}}w_{{1}}w_{
  {3}}-w_{{0}}w_{{2}}w_{{3}}%
\end{split}
$
}
\end{equation*}
and 
\begin{equation*}
  9\,{z}^{4}+7\,b_{{0}}{z}^{3}-8\,{z}^{3}b_{{4}}-6\,{z}^{2}b_{{0}}b_{{4}
  }+5\,b_{{1}}{z}^{2}-4\,zb_{{1}}b_{{4}}+3\,b_{{2}}z-2\,b_{{2}}b_{{4}}+b
  _{{3}}
\end{equation*}
equal to 
\begin{equation*}
9\, \left( {z}^{2}-w_{{0}} \right) ^{2}.
\end{equation*}

Equating coefficient-wise the above polynomials and eliminating $w_0$, $w_1$, $w_2$ and $w_3$ from resulting set of equations allows us to solve
for the $b_i$s, giving 
\begin{equation*}
\begin{split}
&\left( {x}^{4}+{\frac {8\,{x}^{3}b_{{4}}}{7}}+ \left( -{\frac{30}{49}
}+{\frac {36}{49}}\alpha \right) {b_{{4}}}^{2}{x}^{2} \right. \\ 
&+ \left. \left( -{\frac{40}{49}}+{\frac {48}{49}}\alpha \right) {b_{{4}}}^{3}x-{
\frac {{b_{{4}}}^{4} \left( {\frac{59}{7}}+{\frac {156}{7}}\alpha \right) }{343}} \right) ^{2} \left( x-b_{{4}} \right) 
\end{split}
\end{equation*}
where $\alpha = \pm \sqrt{-3}$. The condition for this polynomial to have at most two critical values then imposes $b_4 \ne 0$, so we can take $b_4 = 1$ to obtain the desired two polynomials.
\end{proof}

The family \Fam{10} is defined over $\mathbb{Q}$ as the following lemma
shows.

\begin{lemma}\label{lemma:o10}
For the sporadic family \Fam{10} we have the
polynomials $Q_1(P(x))$ and $Q_2(R(x))$ where
\begin{align*}
    Q_1(x) &= x\left(x-\frac{64}{9}\right) \\
    P(x) &= x^3\left(x^2+\frac{5x}{3}+\frac{40}{9}\right), \\ 
    Q_2(x) &= x\left(x-\frac{4}{9}\right) \\ 
    R(x) &= x^3\left(x-\frac{5}{3}\right)^2.
\end{align*}
\end{lemma}

\begin{proof}
The two trees for the family \Fam{10} are obtained via the compositions $Q \circ P$ and $Q \circ R$ of the dessins $P,Q,R$ depicted in \cref{F:compF10}. 
Note that $Q$ serves to subdivide the edges of the dessins $P,Q$.

We calculate two distinct Shabat polynomials $Q_1(x)$ and $Q_2(x)$ for the dessin $Q$, in order to properly align its vertices with the critical values of the Shabat polynomials for $P$ and $R$, respectively. The polynomials 
$Q_1(x), Q_2(x)$ are obtained by fixing the degree $3$ black vertex to be at $0$. To get the polynomial for $P$, we assume that the positions of the two black leaves are related by a degree $2$-polynomial with unknown coefficients and solve the resulting system.

\begin{figure}[ht]
    \centering
        \includegraphics[width=0.6\linewidth]{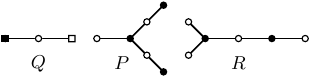}
    \caption{Composition for $F_{10}$.}
    \label{F:compF10}
\end{figure}
\end{proof}

The next family, \Fam{11}, is once again defined over $\mathbb{Q}$.
\begin{lemma}\label{lemma:o11}
For the sporadic orbit \Fam{11} we have the
polynomials $Q_1(P(x))$ and $Q_2(R(x))$ where
\begin{align*}
    Q_1(x) &= x\left(x - \frac{1}{4}\right)\\
    P(x) &= A(B(x)) \\ 
    A(x) &= (x-1)^4\left(x+\frac{1}{4}\right) \\ 
    B(x) &= (4x+4)x+1\\
    Q_2(x) &= x\left(x-\frac{512}{3}\right)\\
    R(x) &= (x^2-3)^4\left(x-\frac{5}{3}-\frac{\sqrt{-2}}{3}\right)\left(x-\frac{5}{3}+\frac{\sqrt{-2}}{3}\right)\\ 
\end{align*}
\end{lemma}
\begin{proof}
Both dessins are obtained as subdivisions (compositions with the 2-star $Q$
depicted in \cref{F:compF10}) of the dessins $P$ and $R$ depicted in 
\cref{F:compF11}. The dessin $A$ is further a composition $A = Q \circ P$ (the corresponding monodromy group being imprimitive), while $R$, interestingly, is not (the corresponding group being primitive). 

As before, two different Shabat polynomials $Q_1(x)$ and $Q_2(x)$ for the $2$-star must be computed to properly compose with $P(x)$ and $R(x)$. The polynomial for $A$ is obtained by placing the  unique degree-$4$ black vertex at $0$, and solving for the position of the remaining black vertex, subject to the restriction of having most two critical values. The Shabat polynomial for $B$ is easily computed from its specification. Finally, the Shabat polynomial for $R$ is computed using the differentiation trick. 

\begin{figure}[ht]
    \centering
        \includegraphics[width=0.7\linewidth]{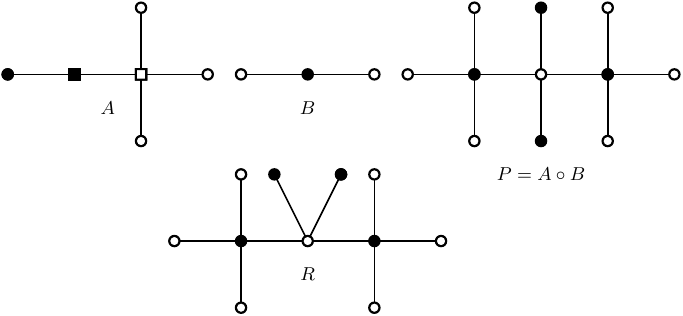}
    \caption{Dessins involved in obtaining the trees of $F_{11}$.}
    \label{F:compF11}
\end{figure}
\end{proof}

The calculation of the Shabat polynomials for \Fam{12}, defined over $\mathbb{Q}(\sqrt{273})$, concludes this subsection.
\begin{lemma}\label{lemma:o12}
For the sporadic orbit $O_{12}$ we have
\begin{footnotesize}
\begin{align*}
    P(x) &= \left(x^2+\frac{2\sqrt{273}}{3}+13\right)^5
    \left(x^3-\frac{13x^2}{3}-x\left(\frac{2\sqrt{273}}{11}-\frac{39}{11}\right)+\frac{91}{15}+\frac{26\sqrt{273}}{45}\right), \\ 
    Q(x) &= x\left(x-\frac{896}{120285}\left(-21 + \sqrt{273}\right)^5 \left(-111+7\sqrt{273}\right)\right),
\end{align*}
\end{footnotesize}
for the first tree, while the second tree is obtained by $\sqrt{273} \mapsto -\sqrt{273}$.
\end{lemma}
\begin{proof}
We omit the proof as it is completely analogous to the ones presented above, involving the use of the differentiation trick and compositions/subdivisions
of dessins.
\end{proof}

\subsection{Monodromy groups}

\begin{lemma}
The monodromy groups for the sporadic orbits ${\mathcal F}_7$ to ${\mathcal F}_{12}$ are as follows
$$\begin{array}{|c|c|}\hline
{\mathbf F_7} & \mathrm{PSL}(3,2)\\ \hline
{\mathbf F_8} & A_7\\ \hline
{\mathbf F_9}  & \mathrm{PSL}(2,8)\rtimes\Z_8\\ \hline
{\mathbf F_{10}^1}  & (A_5\times A_5)\rtimes (\Z_2\times\Z_2)\\ \hline
{\mathbf F_{10}^2}  & (A_5\times A_5)\rtimes \Z_2\\ \hline 
{\mathbf F_{11}^1}  & \Z_2^8\rtimes ((A_5\times A_5)\rtimes D_8)\\ \hline
{\mathbf F_{11}^2} & (A_{10}\times A_{10})\rtimes D_8\\ \hline 
{\mathbf F_{12}}  & (A_{13}\times A_{13})\rtimes \Z_2\\ \hline
\end{array}$$
\end{lemma}

\begin{proof}
These are the result produced using GAP. It is worth mentioning that the two ${\mathbf F_{11}}$ cases are very interesting. One can check that ${\mathbf F_{11}^1}$ has order 7.372.800 and ${\mathbf F_{11}^2}$  has order 26.336.378.880.000.
\end{proof}

Figures \ref{F:family9},  \ref{F:family10}, \ref{F:family11} and  \ref{F:family12} show the two trees for the sporadic dessins $F_9$, $F_{10}$, $F_{11}$ and $F_{12}$ respectively.

\begin{figure}[ht]
    \centering
    \begin{minipage}[b]{0.48\textwidth}
        \includegraphics[width=\linewidth]{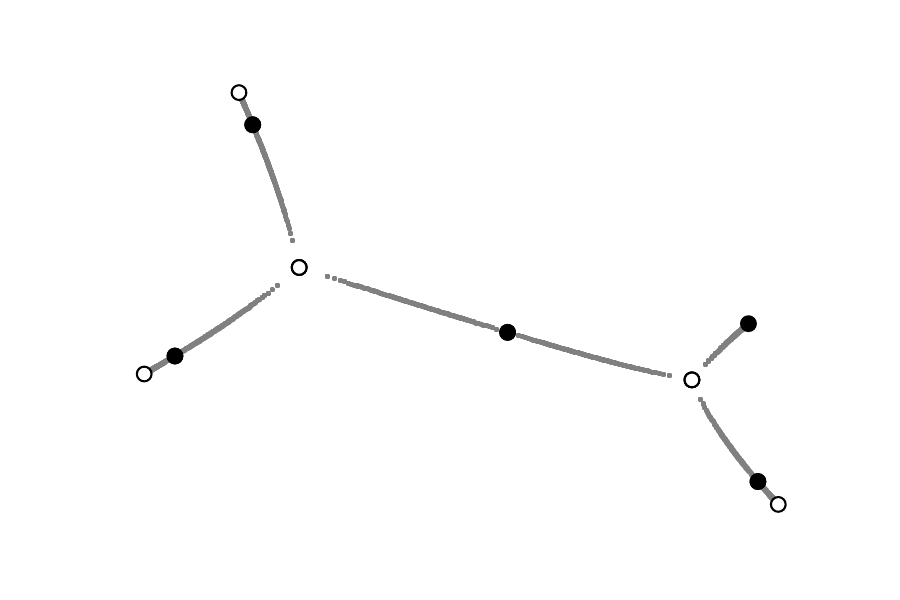}
    \end{minipage}
    \hfill
    \begin{minipage}[b]{0.48\textwidth}
        \includegraphics[width=\linewidth]{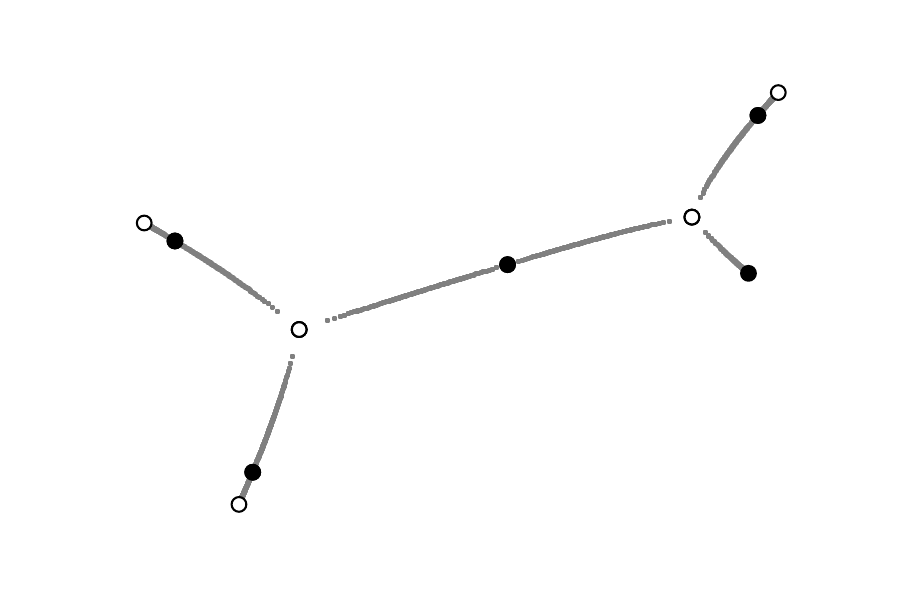}
    \end{minipage}
    \caption{Trees for $F_9$}
    \label{F:family9}
\end{figure}

\begin{figure}[ht]
    \centering
    \begin{minipage}[b]{0.48\textwidth}
        \includegraphics[width=\linewidth]{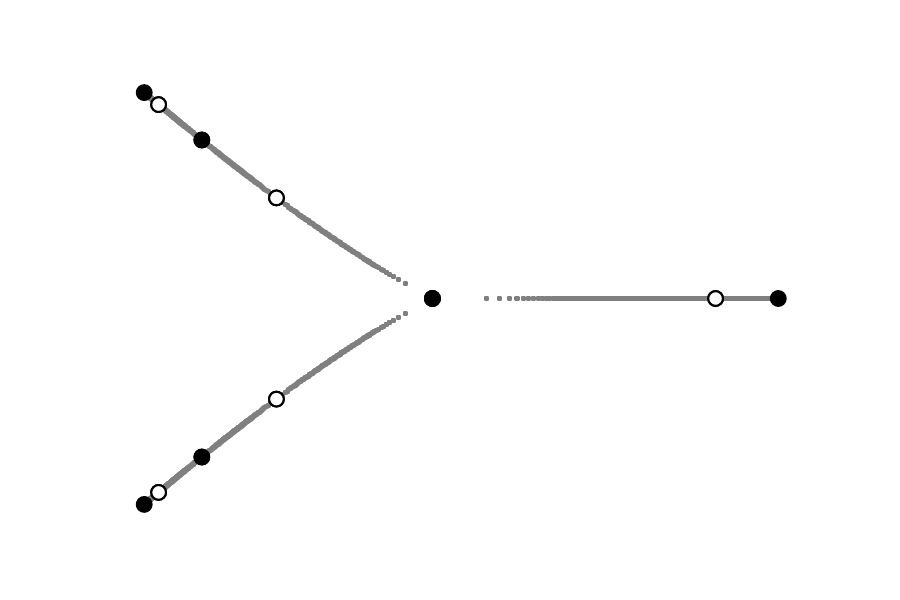}
    \end{minipage}
    \hfill
    \begin{minipage}[b]{0.48\textwidth}
        \includegraphics[width=\linewidth]{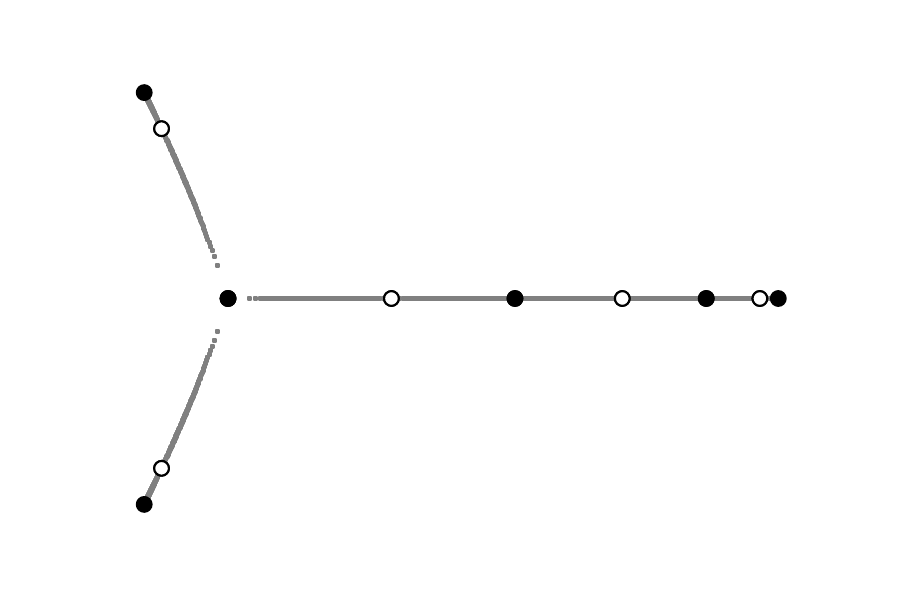}
    \end{minipage}
    \caption{Trees for $F_{10}$}
    \label{F:family10}
\end{figure}

\begin{figure}[ht]
    \centering
    \begin{minipage}[b]{0.48\textwidth}
        \includegraphics[width=\linewidth]{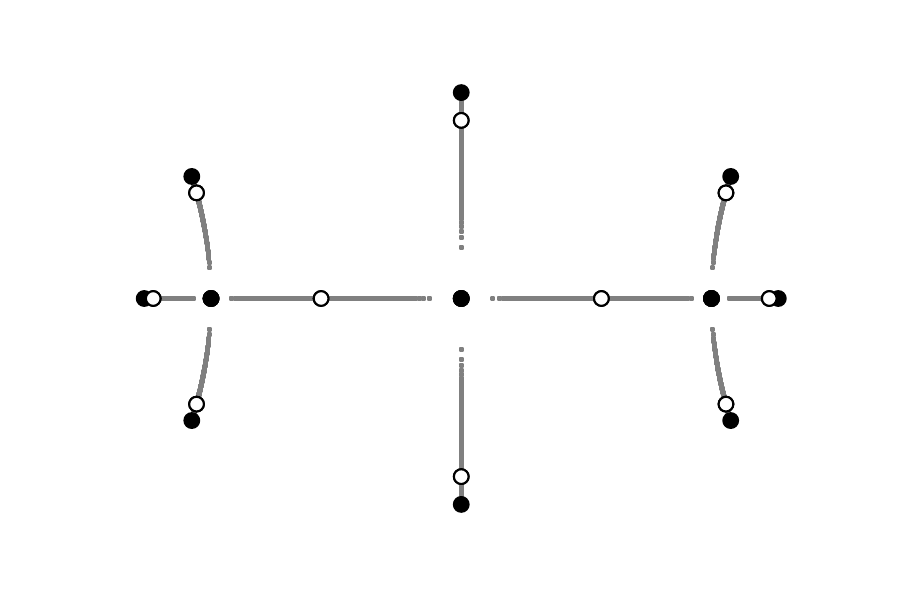}
    \end{minipage}
    \hfill
    \begin{minipage}[b]{0.48\textwidth}
        \includegraphics[width=\linewidth]{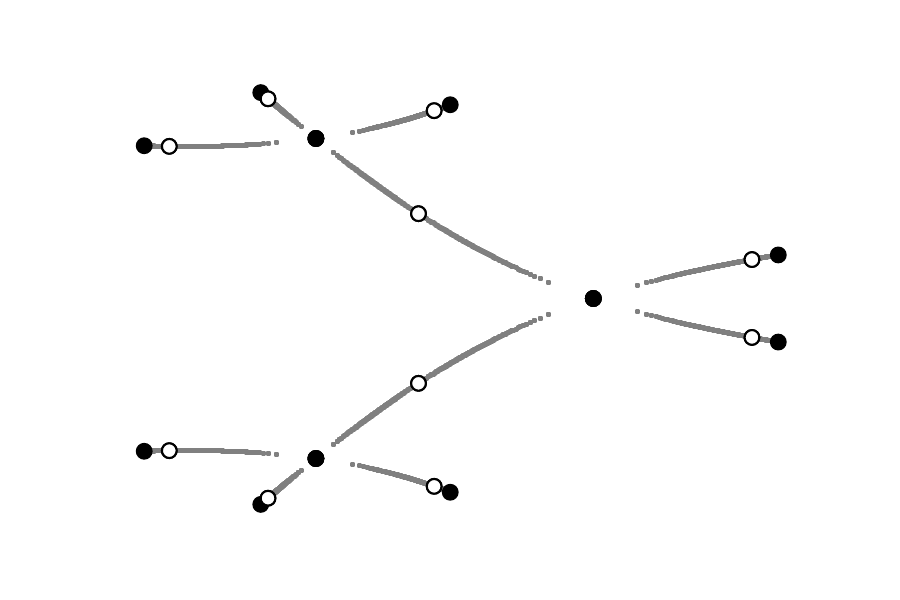}
    \end{minipage}
    \caption{Trees for $F_{11}$}
    \label{F:family11}
\end{figure}

\begin{figure}[ht!]
    \centering
    \begin{minipage}[b]{0.48\textwidth}
        \includegraphics[width=\linewidth]{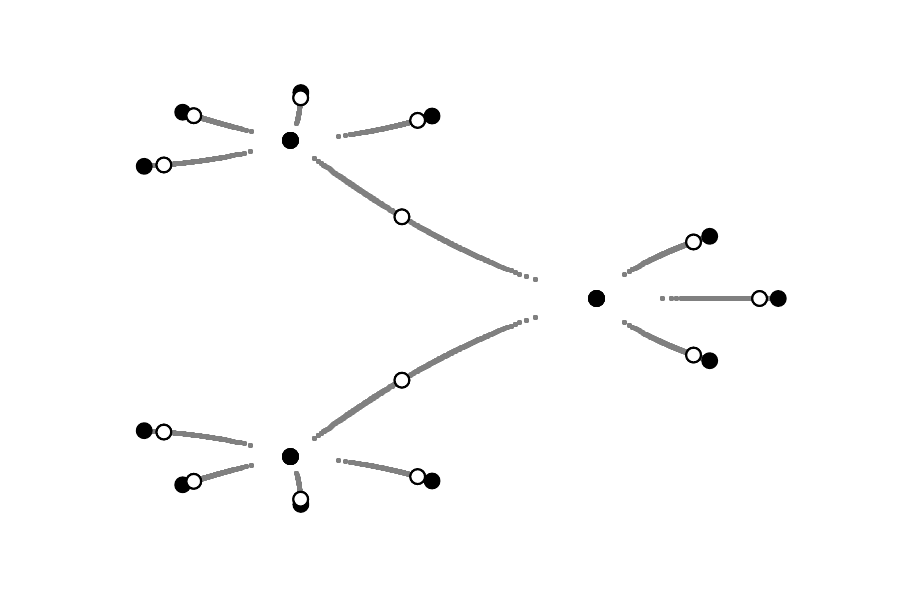}
    \end{minipage}
    \hfill
    \begin{minipage}[b]{0.48\textwidth}
        \includegraphics[width=\linewidth]{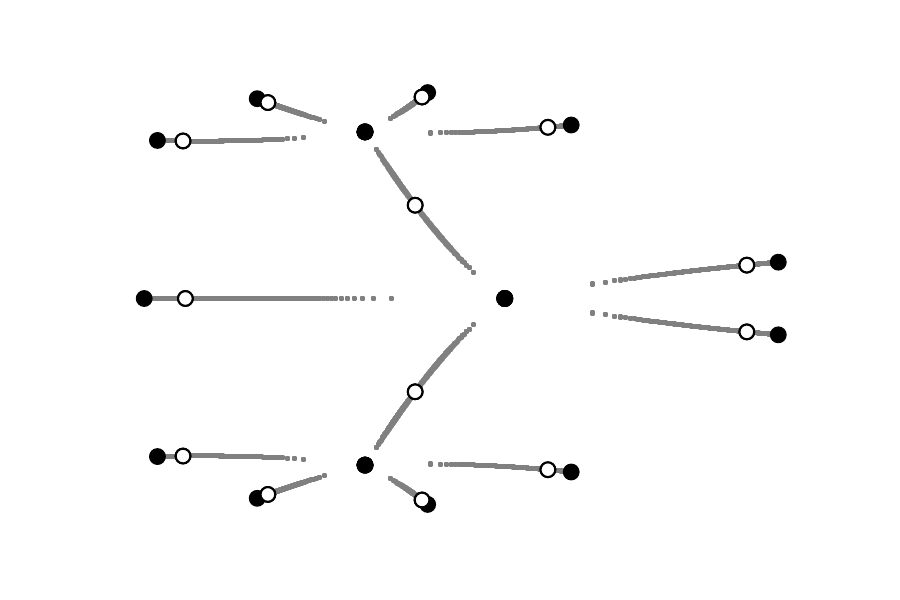}
    \end{minipage}
    \caption{Trees for $F_{12}$}
    \label{F:family12}
\end{figure}

\section{Concluding remarks}

We've determined the Shabat polynomials and monodromy groups
for all families of trees containing exactly two trees. With regards to the action of $\absgal$ on dessins, we summarise our results
in \cref{tab:resume}, mentioning relevant invariants that take different values for each tree in case the family splits.

As for future work, computing Shabat polynomials
and determining the monodromy groups for more small families 
of tree dessins (for example all families of three trees and so on) would help further enrich the catalogues presented in this work and \cite{betremaZvonkinList,adrianov2020davenport,cameron2019shabat,lmfdb}. Such listings of dessins, their Belyi functions, and their associated invariants, serve as a valuable source of (counter)examples for further study.

Inspecting \cref{tab:resume} we notice that the monodromy groups, an
invariant computed via purely combinatorial means, detect almost all
cases of split orbits: only in the case of \Fam{3} do the monodromy
groups coincide for both trees even when the last are defined over a
quadratic field. It would be of interest to compute other combinatorial
invariants for \Fam{3} to see whether they detect the split cases.

\begin{table}[h!]
\centering
\begin{tabular}{|c|c|c|}
\hline
Family & Field of definition & Different groups? \\ \hline
\Fam{1} & $\mathbb{Q}(\sqrt{-rst(r+s+t)})$ &  No. \\ \hline
\Fam{2} & $\mathbb{Q}$ & Yes. \\ \hline
\Fam{3} & $\mathbb{Q}$ or an arbitrary real quadratic field. & No. \\ \hline
\Fam{4} & $\mathbb{Q}(\sqrt{-3})$  & No. \\ \hline
\Fam{5} & $\mathbb{Q}$ & Yes. \\ \hline
\Fam{6} & $\mathbb{Q}(\sqrt{6})$ & No. \\ \hline
\Fam{7} & $\mathbb{Q}(\sqrt{-14})$ & No. \\ \hline
\Fam{8} & $\mathbb{Q}(\sqrt{21})$ & No. \\ \hline
\Fam{9} & $\mathbb{Q}(\sqrt{-3})$ & No. \\ \hline
\Fam{10} & $\mathbb{Q}$ & Yes. \\ \hline
\Fam{11} & $\mathbb{Q}$ & Yes. \\ \hline
\Fam{12} & $\mathbb{Q}(\sqrt{273})$ & No. \\ \hline
\end{tabular}
\caption{Fields of definition for all families of two trees. Quadratic
fields correspond to $2$-orbits while $\mathbb{Q}$ means
the two trees are both fixed points under the action of $\absgal$.
For families splitting into two fixed-points, we list whether
the monodromy groups are different for each tree.}
\label{tab:resume}
\end{table}

\appendix
\section{Shabat polynomial for brushes}
\label{SS:ShabatForBrush}
An $(p,q)$-brush, for $p,q \in \mathbb{N}_{> 0}$, is a bicolored tree depicted as in Figure \ref{fig:rstree2}, having one black vertex of degree $p+1$ and one white vertex of degree $q+1$. In \cite{adrianov2020davenport}, it is shown that such a brush admits a Shabat polynomial that is defined from a Jacobi polynomial, namely
\begin{equation}\label{eq:rrBrush}
    P(x) = \left(\frac{x+1}{2}\right)^{p+1} J_{q}(-q-1,p+1,x)
\end{equation}
and this polynomial always has two critical values: $0$ and $1$.

\begin{figure}[!h]
    \centering
    \includegraphics[width=0.3\linewidth]{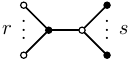}      
    \caption{The tree corresponding to $P$ in \cref{lem:p4}.}
    \label{fig:rstree2}
\end{figure}
For our purposes, we will also need to compute Shabat polynomials for brushes where do not suppose that the critical values are $0$ and $1$. This is done in Sections \ref{SS:ShabatO5} and \ref{SS:ShabatO6}, where we need to change the critical values to obtain a good composition. To do so, we reason that since there is a black vertex of degree $p+1$ and a white vertex of degree $q+1$, $P$ should contain a factor $(x-b_0)^p$ and $P+1$ should contain a factor $(x-w_0)^q$, along with other degree $1$ factors. The derivative of $P$, which is of degree $p+q$, should then contain factors of both $(x-b_0)^p$ and $(x-w_0)^q$. We can furthermore impose that $w_0$ and $b_0$ are symmetric, so that $P'(x) = K(x+a)^p(x-a)^q$ for some complex $a$ and multiplicative constant $K$. Therefore, we obtain that
\begin{equation}
    P(x) = K \int (x+a)^p (x-a)^q \, \mathrm{d}x + C,
\end{equation} 
for some constant $C$.
Let us now fix $a = 1$, and compute this integral where $p = q$, which is the interesting case for $O_5$ (section \ref{SS:ShabatO5}) and $O_6$ (section \ref{SS:ShabatO6}). We have 
\begin{align*} (x+a)^p (x-a)^p & = (x^2 - a^2)^p \\
& = \sum\limits_{k=0}^p (-1)^k \binom{p}{k} x^{2(p-k)} \\
\end{align*}
and thus the integral is equal to
\begin{equation*}
 K \sum\limits_{k=0}^p (-1)^k \binom{p}{k} \frac{x^{2p-2k+1}}{2p-2k+1} + C.
\end{equation*}
The only critical points of $P$ are $1$ and $-1$, and one then checks that the critical values of $P$ are of the form
\begin{equation*}
  K (-1)^i \left( \frac{(2p)!!}{(2p+1)!!} \right) + C,
\end{equation*}
with $i = p$ or $i = p+1$. 

If we want the critical values to be $0$ and $c$,
we must take $C = c/2$. In the case of the family \Fam{5}, we want the critical values to be $0$ and $-1$, so we choose 
\begin{equation}\label{eq:brushForF5}
    C = - \frac{1}{2}, \quad K = \frac{(-1)^{p+1} (2p+1)!!}{2(2p)!!}.
\end{equation} 

In the case of the family \Fam{6}, we want the critical values to be fifth roots of unity, so that they align with the position of the black vertices of a $5$-star with a white center, so we choose
\begin{equation}\label{eq:brushForF6}
 C = - \frac{1}{4} \sqrt{5} - \frac{1}{4}, \quad K = \frac{(-1)^{p} (2p+1)!}{2^{2p+2} p!^2} \sqrt{-10+2\sqrt{5}}.
\end{equation}

\bibliographystyle{plain}
\bibliography{biblio} 

\end{document}